\documentclass[journal]{IEEEtran}
\usepackage{CJK}
\usepackage{amssymb}
\usepackage{epstopdf}
\usepackage{graphicx}
\usepackage{multirow}
\usepackage{diagbox}
\usepackage{algorithmicx,algorithm}
\usepackage{framed}
\usepackage{booktabs}
\usepackage{threeparttable}
\usepackage{comment}
\usepackage{amsmath}
\usepackage{xcolor}

\usepackage{bm}
\usepackage{subfigure}

\usepackage{amsthm}

\graphicspath{ {./images/} }

\newtheorem{thm}{Theorem}
\newtheorem{lem}{Lemma}
\newtheorem{cor}{Corollary}
\newtheorem{pro}{Proposition}

\newtheorem{rmk}{Remark}
\newtheorem{defi}{Definition}
\newtheorem{ass}{Assumption}

\newcommand {\emptycomment}[1]{}

\newcommand{\be }{\begin{equation}}
\newcommand{\ee }{\end{equation}}

\newcommand{\noi}{\noindent}

\newcommand{\ttt}{\theta}
\newcommand{\wh}{\widetilde}

\newcommand{\f}{\frac}

\newcommand{\aaa}{\alpha}

\newcommand{\nn}{\langle}
\newcommand{\mm}{\rangle}
\newcommand{\nono}{\nonumber}
\newcommand{\huaX}{\mathcal{X}}

\newcommand{\mbb}{\mathbb}

\newcommand{\ooo}{\omega}
\newcommand{\kkk}{\kappa}

\def\bea{\begin{eqnarray}}
\def\eea{\end{eqnarray}}
\def\be{\begin{equation}}
\def\ee{\end{equation}}
\def\blm{\begin{lem}}
\def\elm{\end{lem}}
\def\btm{\begin{theorem}}
\def\etm{\end{theorem}}

\def\ff{\Phi}

\newcommand{\huaB}{\mathcal{B}}
\newcommand{\huaS}{\mathcal{S}}

\newcommand{\huaR}{\mathcal{R}}
\newcommand{\huaE}{\mathcal{E}}
\newcommand{\huaF}{\mathcal{F}}
\newcommand{\huaG}{\mathcal{G}}

\newcommand{\huaV}{\mathcal{V}}

\newcommand{\huaQ}{\mathcal{Q}}
\newcommand{\huaP}{\mathcal{P}}

\newcommand{\huaK}{\mathcal{K}}
\newcommand{\huaO}{\mathcal{O}}
\newcommand{\huaT}{\mathcal{T}}
\newcommand{\huaZ}{\mathcal{Z}}

\newcommand{\huaM}{\mathcal{M}}

\def\bea{\begin{eqnarray}}
\def\eea{\end{eqnarray}}
\def\be{\begin{equation}}
\def\ee{\end{equation}}
\def\blm{\begin{lem}}
\def\elm{\end{lem}}

\def\bea{\begin{eqnarray}}
	\def\eea{\end{eqnarray}}
\def\be{\begin{equation}}
	\def\ee{\end{equation}}
\def\blm{\begin{lem}}
	\def\elm{\end{lem}}
\def\bes{\begin{eqnarray*}}
	\def\ees{\end{eqnarray*}}
\def\beal{\begin{aligned}}
	\def\eeal{\end{aligned}}
\def\ppp{\Phi}
\def\wh{\widehat}
\def\ff{\Phi}

\def\nb{\nabla}

\def\ww{\widetilde}
\def\sss{\sigma}

\newtheorem{theorem}{Theorem}

\allowdisplaybreaks

\begin{document}
\begin{CJK*}{GBK}{song}

\title{High-Probability Convergence Theory for Distributed Composite Optimization with Sub-Weibull Noises}
\author{Zhan Yu, Zhongjie Shi, Deming Yuan, \IEEEmembership{Senior Member, IEEE} \thanks{This work is in part supported by  the National Natural Science Foundation of China, under Grant No. 12401123, 62373190 and the Research Grants Council of Hong Kong, under Grant HKBU 12301424. (\emph{Corresponding Author: Zhan Yu})
		
		Z. Yu is with the Department of Mathematics, Hong Kong Baptist University (mathyuzhan@gmail.com). Z. Shi is with School of Mathematics, Georgia Institute of Technology (zshi332@gatech.edu). D. Yuan is with the School of Automation, Nanjing University of Science and Technology (dmyuan1012@gmail.com).}}

\maketitle

\begin{abstract}
With the rapid development of  distributed optimization (DO) theory, the distributed stochastic gradient methods (DSGMs) occupy an important position. Although the theory of different DSGMs has been widely established, the main-stream results of existing work are still derived under the condition of light-tailed stochastic gradient noises. Increasing examples from various fields, indicate that, the light-tailed noise model is overly idealized in many practical instances, failing to capture the complexity and variability of noises in real-world scenarios, such as the presence of outliers or extreme values from data science and statistical learning.  To address this issue, we propose a new DO framework that incorporates stochastic gradients under sub-Weibull randomness. We study a distributed composite stochastic mirror descent scheme with sub-Weibull gradient noise (DCSMD-SW)  for solving a convex distributed composite optimization (DCO) problem over the time-varying multi-agent network.  By investigating sub-Weibull randomness in DCSMD for the first time, we show that the algorithm is applicable in some common heavier-tailed noise environments while also guaranteeing good convergence properties. We comprehensively study the convergence performance of DCSMD-SW. Satisfactory high-probability convergence rates are derived for DCSMD-SW without any smoothness requirement. The work also offers a unified analytical framework for several critical cases of both algorithms and noise environments.
\end{abstract}

\section{Introduction}
Over the past decade, theory of multi-agent distributed optimization (DO) has experienced explosive development, leading to a multitude of DO algorithms that promote the advancement of fields of control, optimization and learning (\cite{s1}-\cite{lwzw2023}).
A typical feature of DO is that agents collaboratively address the problem  of minimizing a global objective function,   represented as the sum of several local cost functions through effective network communication.   Recently, 
 many fundamental DO approaches have been developed, for example, distributed subgradient method \cite{s1}, distributed dual-averaging  \cite{jam2012}, distributed subgradient-push   \cite{no2015}, distributed mirror descent   \cite{yhhj2018}, distributed primal-dual  \cite{ylyxcj2020}, distributed ADMM  \cite{mo2017}, distributed gradient tracking \cite{pn2021} and DO with variance reduction \cite{xkk2020}. This paper primarily concentrates on the study of consensus-based DSGMs for solving convex DO problem 
(e.g. \cite{yyw2019}, \cite{jam2012}, \cite{yhhj2018}, \cite{rnv2010}, \cite{yhhx2020}, \cite{yhy2022}, \cite{qlxc2025}, \cite{llfh2023}, \cite{lwzw2023}). DSGMs involve algorithms that incorporate randomness of (sub)gradient in their search for optimal solutions. Unlike deterministic methods that follow a fixed path, DSGMs can often explore the solution space more broadly, potentially avoiding local minima or saddle points. This flexibility allows them to better navigate complex landscapes, including large-scale optimization problems.  

In DSGMs, the type of stochastic noise is crucial in the formulation of the stochastic gradient. Different noise characteristics can significantly influence the convergence behavior of DO algorithms. In the early development of DSGMs for DO or DCO, gradient noise was typically assumed to have a light-tailed noise with bounded variance, like sub-Gaussian noise. These methods include distributed stochastic dual-averaging \cite{jam2012}, distributed stochastic gradient descent (DSGD) \cite{rnv2010}, \cite{lwzw2023}, distributed stochastic mirror descent (DSMD) \cite{yhhj2018},  \cite{llww2018}, \cite{xzhyx2022} distributed stochastic gradient-tracking \cite{pn2021}, distributed stochastic gradient-push \cite{no2016}, distributed  stochastic gradient extrapolation \cite{lz2018}, bandit online distributed composite zeroth-order mirror descent \cite{yhhx2020}. These work were established based on the typical light-tailed boundedness condition (e.g. second moment condition or sub-Gaussian condition) for the stochastic gradient or its variance.
An increasing number of instances in machine learning and statistical learning demonstrate that light-tailed noise model is overly idealized for modeling the stochastic environment of learning and optimization (e.g. \cite{ssg2019},  \cite{zhang2020}) and they often fall short in capturing the complexity and variability of noises (e.g. \cite{yfsz2024}-\cite{ghs2018}). It is also quite common that 
deep learning scenarios often exhibit heavy-tailed noises (see e.g. \cite{ghw2020}) and their relevant noise environments are better characterized by heavier tailed noise. 
In existing literature, the studies on DO and its related distributed Nash equilibrium seeking under heavy-tailed noise have only very recently begun to attract attention (e.g. \cite{qlxc2025}-\cite{ylw2025}, \cite{scwwy2025}). Specifically, \cite{qlxc2025} proposed a distributed clipped stochastic dual-averaging for solving DO problem, the high-probability convergence is achieved for the local ergodic sequence. \cite{szcy2025} achieved the almost sure convergence of the state variable of a clipped DSGD for both convex and strongly convex setting. \cite{ylw2025} studied the online DO, consensus-based regret analysis was established for a clipped DSGD. In \cite{scwwy2025}, a distributed Nash equilibrium seeking law is provided based on techniques of gradient clipping.  Compared to the aforementioned works based on light-tailed noise framework, these works    focus on heavy-tailed scenarios where the noise possesses lower-order moments (potentially having unbounded variance) much weaker than the above bounded second-moment condition. 

We observe that, in existing works, despite the introduction of techniques like clipping or normalization, the convergence performance is clearly worse than the optimal or suboptimal rates. This fact drives us to seek a scenario where the stochastic gradient model in DSGMs can exhibit potentially  heavier-tailed feature, while also achieving significantly better convergence performance than the existing results on DSGMs under general heavy-tailed noise condition. 
To this end,  we study DSGMs with stochastic gradient modeled by the sub-Weibull noises. Sub-Weibull noise is a type of random noise characterized by exponentially decaying tails, as nicely discussed in the existing optimization and learning literature \cite{bmcd2024}-\cite{kmd2022}, which provide inspiring  insights. Sub-Weibull distributions are known for their strong tail, concentration and martingale properties for constructing robust confidence bounds in statistics and optimization (e.g. \cite{mdb2024}, \cite{vgna2020}). The sub-Weibull noise can be viewed as a nice intermediate  model between light-tailed and heavy-tailed noise, heavier than sub-Gaussian noise while also serving as a representative lighter subclass of the bounded $p$th-moment noise for $p\in(1,2]$. By using a simple tail parameter $\ttt$ to characterize sub-Weibull randomness, it provides a unified analytical framework for handling and modeling a very broad  typical random noises from sub-Gaussian distribution ($\ttt=\f{1}{2}$), sub-Exponential distribution ($\ttt=1$) to heavier-tailed distributions ($\ttt>1$) as well as the random variables with finite support (e.g. \cite{kmd2022}). Moreover, the intermittent updates can be modeled by sub-Weibull variables, as mentioned in \cite{kmd2022}. Hence, sub-Weibull class possesses strong modeling capability for real-world noises, and provides a good entry point for noise research, avoiding the extremes of being too narrow (such as limited to sub-Gaussian) or too broad (such as including all possible distributions with rare extreme cases), and finds a well-balanced point that can enhance both theoretical generality and practical applicability (e.g. \cite{bmcd2024}). It is noteworthy that the high-probability convergence theory  has not yet been explored for DSGMs with sub-Weibull stochastic gradient noise in the context of convex DCO. This presents a promising opportunity to establish high-probability convergence theory for some classic DSGMs under sub-Weibull noises. We aim to fill this gap and present an initial exploration into the theoretical study of   convex DCO within the framework of sub-Weibull randomness.

In this work, we focus on the  convex DCO problem over a time-varying multi-agent system. We study a distributed composite stochastic mirror descent (DCSMD-SW) algorithm that is utilized for solving the problem in the environment of sub-Weibull randomness. Each agent makes decisions solely based on the sub-Weibull gradient noise estimate of its associated local objective function and the state information received from its neighbors at the corresponding time instant. DSMD has been thriving for over a decade (e.g. \cite{yhhj2018}, \cite{yhhx2020}, \cite{llww2018}, \cite{xzhyx2022}, \cite{fzy2024}). However, all these works are still limited to analysis in the framework of convergence in expectation under light-tailed noises. This study focuses on the high-probability convergence of DCSMD under sub-Weibull noise. Obtaining high-probability bounds related to the difference between the stochastic gradient estimate and the true gradient becomes challenging, as this difference cannot be eliminated like in the analysis of convergence in expectation due to the unbiased condition (\cite{lwzw2023}). Moreover, the influence of sub-Weibull noise on the network communication provides additional difficulties to the high-probability analysis. This work aims to overcome these challenges and provide a new high-probability convergence theory for DCSMD-SW. 

We summarize the main contribution of this work:
\begin{itemize}
	\item In this work, sub-Weibull stochastic gradient noise is  investigated for the first time in the literature of convex DCO.  Under basic assumptions on time-varying network structure  and subgradient boundedness of objective functions (e.g. \cite{rnv2010}), for any confidence level $0<\delta<1$, DCSMD-SW is able to achieve a high-probability convergence rate of $\huaO\big((\log\f{T}{\delta})^{2\ttt}\f{(\log T)^2}{\sqrt{T}}+\left(\log\f{1}{\delta}\right)^2\left(\log\f{T}{\delta}\right)^{2\max\{0,\ttt-1\}}\f{1}{\sqrt{T}}\big)$,  $\ttt\geq1/2$, provided that a diminishing stepsize $\aaa_t=\huaO(\f{1}{\sqrt{t}})$ is utilized. Under a known time horizon $T$, via a constant stepsize strategy, we can achieve a high-probability convergence rate of $\huaO\big ([(\log\f{T}{\delta})^{2\ttt}+(\log\f{1}{\delta})^{2}]\f{1}{\sqrt{T}}\big)$, $\f{1}{2}\leq\ttt<1$; $\huaO\big((\log\f{T}{\delta})^{2\ttt}\f{1}{\sqrt{T}}\big)$,  $\ttt\geq1$.   
	This is also the first time that explicit high-probability convergence rates have been obtained in  convex DCO under a sub-Weibull noise environment (optimal only up to poly-logarithmic factor). 
	Additionally, this work clearly demonstrates how sub-Weibull noise affects information communication among nodes in the multi-agent system. Our framework does not require any extra truncation schemes of stochastic gradients such as gradient clipping, 
	which can potentially complicate parameter tuning 
	in practice (e.g. \cite{ep2024}, \cite{wor2025}, \cite{hfh2025}).

\item The high-probability convergence rate of DCSMD-SW is derived without assuming any smoothness. Moreover, due to the flexible selection of Bregman divergence and local composite regularization functions, our DCSMD-SW analytical framework can degenerate into some crucial distributed algorithms with sub-Weibull randomness, of which the typical representative is the  distributed SGD (DSGD-SW).  
Although the high-probability convergence rates of DSGD-SW have not been studied independently, they have been directly obtained as a special case of this work. Our theory offers a unified high-probability convergence framework for these algorithms under sub-Weibull noises with different tail parameters $\ttt$ that model a broad class of random noises.

\item In existing research on different types of distributed mirror descent (DMD) algorithms for convex optimization e.g. \cite{yhhj2018}, \cite{yhhx2020}-\cite{fzy2024}, the convergence analysis essentially relies on the boundedness of the decision space, which is a key technical assumption. Our work successfully eliminates this requirement, enabling all results to be applicable to DO with unbounded decision domains, including cases of unconstrained DO. Therefore, these results have a clearly broader range of applications.
\end{itemize}

\noi \textbf{Notation:} Denote  ($\mathbb R^n$, $\|\cdot\|$) as the  $n$-dimensional vector space $\mbb R^n$ equipped with an arbitrary norm $\|\cdot\|$. We use $\nn\cdot,\cdot\mm$ to denote the standard Euclidean inner product. Denote by $\|\cdot\|_*$ the dual norm of $\|\cdot\|$, that is $\|\cdot\|_*=\sup_{\|u\|\leq1}|\nn\cdot,u\mm|$.
For a matrix $M\in\mathbb R^{m\times m}$, denote the element in $i$th row and $j$th column by $[M]_{ij}$. For an $n$-dimension Euclidean vector $x$, denote its $i$-th component by $[x]_i$, $i=1,2,...,n$. We use $\mbb P(S)$ to denote the probability of a measurable set $S$. For a convex function $f:\huaX\rightarrow\mbb R$, we use $\partial f(x)$ to denote the subdifferential set of $f$ at $x\in \huaX$. For any $a,b\in\mbb R$, we denote $a_+=\max\{0,a\}$ and $a\vee b=\max\{a,b\}$. For any positive integer $a$, we denote $[a]=\{1,2,...,a\}$. We use $\mbb N_+$ to denote positive integer set.

\section{Problem setting}
This paper considers convex DCO problem over a time-varying multi-agent network. The agents are indexed by $i=1,2,..., m$.  The communication topology among the agents is modeled as a time-varying graph $\huaG_t=(\huaV,\huaE_t, W_t)$, where $\huaV=\{1,2,..., m\}$ is the node set of the multi-agent system, $W_t$ is the communication matrix corresponding to the graph structure at time $t$, $\huaE_t$ is
the set of activated links (the edge set) at time $t$. $W_t$ is the communication matrix associated with the graph $\huaG_t$ such that $[W_t]_{ij}>0$ if $(j,i)\in\huaE_t$ and $[W_t]_{ij}=0$ otherwise, namely, $\huaE_t=\{(j,i)\in\huaV\times\huaV|[W_t]_{ij}>0\}$. We  call $\{j\in\huaV|(j,i)\in\huaE_t\}$ the neighbor set of agent $i\in\huaV$ at time $t$. 

We study the following convex DCO problem
\begin{eqnarray} \label{problem}
	\min_{x\in \huaX} F(x)=\sum_{i=1}^m\Big[f_i(x)+\psi_i(x)\Big].
\end{eqnarray}
In \eqref{problem}, $x\in \huaX\subset\mathbb R^n$ is a global decision vector, $\huaX$ is a closed convex decision domain (that can be unbounded). The function $f_i: \huaX\to \mathbb R$ is the local convex cost (or loss) function known only at the $i$th agent. $\psi_i:\huaX\rightarrow \mbb R$ is a simple convex regularization function associated with the agent $i$,  known only at the $i$th agent. In this paper, all the objective functions can
 be nonsmooth.  We suppose that there is at least an optimal point $x^*\in\huaX$ such that $F(x)\geq F(x^*)$ for $x\in\huaX$. 

	The composite regularization terms $\psi_i$, $i=1,2,...,m$ are essential components of the problem \eqref{problem}. They are typically used to promote various types of solution structures or to capture the complexity of the solutions. The components $\psi_{i}:\huaX\rightarrow\mbb R$ can represent various regularizers for practical purposes. For example, they may include the indicator function $I_{\huaX}$ of the decision set $\huaX$; the $l_1$-regularizer $\lambda\|x\|_{1}$ with $\lambda>0$  which promotes the sparsity of solutions in distributed estimation within sensor networks; and the mixed regularizer $\f{\lambda_1}{2}\|x\|_2^2+\lambda_2\|x\|_{1}$ with $\lambda_1>0$, $\lambda_2>0$, used in distributed elastic net regression problems (e.g. \cite{yhhx2020}). 

 Before coming to define the sub-Weibull randomness on the stochastic gradient noises, let us first recall the definition of sub-Weibull random variable below (Definition 5 in \cite{mdb2024}).
 \begin{defi}
 	For $\ttt>0$ and $\kkk>0$, a random variable $X$ is called sub-Weibull ($\ttt,\kkk$) if it satisfies
 	$\mbb E [\exp(|X|/\kkk)^{1/\ttt}]\leq2$.
 \end{defi}
 
 For handling the stochasticity, let us define the $\sss$-field generated by the entire history of the randomness till step $t$ by $\huaF_t$. 
  At each iteration step $t\geq1$, each agent $i\in\huaV$ has access to a sub-Weibull noisy gradient $\widehat g_{i,t}^{SW}\in\mbb R^n$ at point $x_{i,t}\in\huaX$ $s.t.$ $$\wh g_{i,t}^{SW}=g_{i,t}-\xi_{i,t},$$ where $g_{i,t}\in\partial f_{i}(x_{i,t})$ and $\mbb E[\xi_{i,t}|\huaF_{t-1}]=0$. 
 The following norm sub-Weibull stochastic noise assumption is made for the stochastic noise $\xi_{i,t}$. It pertains to a norm sub-Weibull condition on the stochastic gradient noise.
 \begin{ass}\label{subweibull_ass}
 	For some $\ttt\geq1/2$, there exists a constant $\kkk>0$ such that for each iteration step $t\geq1$, $\|\xi_{i,t}\|_*$, $i\in\huaV$ is Sub-Weibull ($\ttt$,$\kkk$) (denoted by $\|\xi_{i,t}\|_*\sim SubW(\ttt,\kkk)$) conditioned on $\huaF_{t-1}$, namely 
 	\bea
 	\mbb E\left[\exp\left\{\left(\f{\|\xi_{i,t}\|_*}{\kkk}\right)^{1/\ttt}\right\}\Big|\huaF_{t-1}\right]\leq2.
 	\eea
 \end{ass}

 Classical SMD theory (e.g.  \cite{lz2020}, \cite{nl2014}) indicates that, for mirror descent on $\mbb R^n$ equipped with an arbitrary norm $\|\cdot\|$, the stochastic gradient $\wh g_{i,t}^{SW}$ and subgradient $g_{i,t}$ are naturally performed in the dual space ($\mbb R^n,\|\cdot\|_*$), hence it is natural to utilize the dual norm of its noise $\xi_{i,t}=g_{i,t}-\wh g_{i,t}^{SW}$ to measure sub-Weibull  stochasticity related to  $\wh g_{i,t}^{SW}$. Accordingly, our framework can encompass different non-Euclidean metrics, rather than just the Euclidean metric  (at this time, $\|\cdot\|=\|\cdot\|_*=\|\cdot\|_2$). This type of dual norm definition on stochastic gradient has been widely adopted in the frameworks for light-tailed (e.g. \cite{jam2012}, \cite{nl2014}) and heavy-tailed stochasticity (e.g. \cite{ep2024}, \cite{nnen2023}) along different research lines of SMD.
The sub-Weibull randomness has well been  investigated only in several recent works from different optimization perspectives (\cite{bmcd2024}-\cite{kmd2022}). $\ttt$ is the core tail parameter determining the heaviness of the noises. A larger value of \(\theta \) indicates heavier tails (sub-Gaussian ($\ttt=\f{1}{2}$)$\to$sub-Exponential ($\ttt=1$)$\to$heavier-tailed regime (sub-Weibull with $\ttt>1$)). In addition to the above sub-Weibull condition measured by exponential moment, there is another research line based on a more general bounded $p$th-moment condition, namely $\mbb E[\|\xi_{i,t}\|_*^p|\huaF_{t-1}]\leq\nu^p$, for some $\nu>0$ and $p\in(1,2])$ (e.g. \cite{qlxc2025}-\cite{ylw2025}, \cite{scwwy2025}, \cite{nnen2023}). Bounded $p$th-moment  characterizes a relatively broad class of heavy-tailed noise, while sub-Weibull class (as a representative subclass with better concentration and martingale properties) provides an operationally more refined and targeted noise model framework that also exhibits heavier-tailed features (when $\ttt>1$) than commonly encountered noise such as sub-Gaussian and sub-Exponential noise. In the context of DCO, there is a noticeable lack of relevant research. We will provide a starting point for the study of  DCSMD-SW  for solving DCO problem \eqref{problem} in the context of sub-Weibull randomness.

Now, we come to describe the network environment. In this paper, we focus on the following graph class.
\begin{ass}\label{as2} (a) The communication matrix $W_t$ is a doubly stochastic matrix, i.e. $\sum_{i=1}^m[W_t]_{ij}=1$ and $\sum_{j=1}^m[W_t]_{ij}=1$ for any $i$ and $j$. (b) There exists an integer $B\geq1$ $s.t.$ the graph $(\huaV, \huaE_{kB+1}\cup\cdots\cup\huaE_{(k+1)B})$ is strongly connected for any $k\geq0$. (c) There exists a scalar $0<\eta<1$ $s.t.$ $[W_t]_{ii}\geq \eta$ for all $i$ and $t$, and $[W_t]_{ij}\geq\eta$ if $(j,i)\in \huaE_t$.
\end{ass}

Assumption \ref{as2} is a commonly adopted assumption on the information exchange of the multi-agent networks in existing work (e.g. \cite{rnv2010}, \cite{yhy2022}, \cite{lwzw2023}). It is important to  guarantee the consensus of the time-varying  network. $W_t$ is utilized to denote the communication pattern among all the nodes at time $t$. (b) indicates that the network is frequently connected, but it does not need to be connected at each instant. The model is more general than fixed connected networks and it includes them as a special case of $B=1$. (c) guarantees that each agent gives a
	sufficient weight to its current iterate and all the iterates it receives. We remark that our current work primarily focuses on networks where a doubly stochastic weight matrix can be constructed. Although any strongly connected digraph can be made doubly
	stochastic after adding enough number of self-loops and all undirected regular graphs are doubly stochasticable (see \cite{lwzw2023}, Corollaries 4.2 and 4.3 in \cite{gc2010}), for a highly general directed graph, constructing such a weighted matrix is challenging (\cite{gc2012}). In the future, it would be meaningful to combine some techniques and algorithms designed for directed graphs (e.g. \cite{no2015}) to relax the doubly stochastic assumption and establish convergence analysis under sub-Weibull randomness.

Denote the transition matrices of the communication matrix $W_t$ by $\Psi(t,s)=W_tW_{k-1}\cdots W_s, t\geq s\geq0$. We require an important fact about the transition matrices. The following fundamental result underlies the key relationship between network topology and the subsequent convergence analysis.
\begin{lem}\label{network_nedic_lem}
(Lemma 3.2 in Reference \cite{rnv2010})	 Let Assumption \ref{as2} hold, then for all $i,j\in \huaV$ and all $t,s$ satisfying $t\geq s\geq 0$, it holds that $\left|[\Psi(t,s)]_{ij}-\f{1}{m}\right|\leq \ooo\gamma^{t-s}$,
	in which $\ooo=(1-\f{\eta}{4m^2})^{-2}$ and $\gamma=(1-\f{\eta}{4m^2})^{\f{1}{B}}$.
\end{lem}

We consider the following assumption on the local cost functions $\{f_i\}_{i\in\huaV}$ and local regularization functions $\{\psi_i\}_{i\in\huaV}$.
\begin{ass}\label{gradient_bdd}
	There is a constant $G>0$  $s.t.$ for any $x\in\huaX$ and $g\in\partial f_i(x)$, $i\in\huaV$, $\|g\|_*\leq G$.  The  functions $\psi_i$, $i\in\huaV$ are convex and $G_\psi$-Lipschitz over $\huaX$ with a constant $G_\psi>0$, i.e.
	$|\psi_i(x)-\psi_i(y)|\leq G_\psi\|x-y\|$, $x,y\in\huaX$.
	\end{ass}

The fundamental assumption related to subgradient boundedness of local objective functions has been widely considered since the inception of DO \cite{s1}, \cite{jam2012}-\cite{yhhj2018}, \cite{rnv2010}, \cite{yhhx2020}, \cite{llww2018}, \cite{lj2021}-\cite{fzy2024}. Different types of local cost functions satisfies the assumption, including, e.g. Nesterov's nonsmooth test function (\cite{yhy2022}) $|[x]_1-1|+\sum_{k=1}^{n-1}|1+[x]_{k+1}-2[x]_k|$; logistic regression function (\cite{xzhyx2022}) $\log(1+\exp(-b\nn a,x\mm))$, $(a,b)\in\mbb R^n\times\{-1,1\}$; the  hinge loss function (\cite{jam2012}) $[1-b\nn a,x\mm]_+$, $(a,b)\in\mbb R^n\times\{-1,1\}$ in many common constraint sets such as box constraint, ball constraint, simplex constraint. In these settings, all the aforementioned regularization functions in this section satisfy this assumption and can be used for different practical purposes. The difficulty in totally removing  Assumption \ref{gradient_bdd} lies in its core role in the analysis, particularly in relation to estimates associated with convexity and it deeply influences the consensus bounds and many estimates related to network communication as witnessed in subsequent analysis. A main restrictiveness of this assumption is that, although the convergence rates achievable for nonsmooth functions under this assumption are already quite favorable, for smooth functions, the  rates obtained based on this assumption may be relatively conservative. For this case, a possible alternative assumption is the $L$-smoothness condition, which has the potential to establish results with faster rates when combined with approaches such as Nesterov's momentum in the future.

The following distance-generating function $\Phi$ and Bregman divergence $D_\Phi(\cdot\|\cdot)$ is crucial for formulating the main algorithm (see e.g. \cite{yhhj2018}, \cite{yhy2022}, \cite{llww2018},  \cite{lz2020}).
\begin{defi}\label{defiphi}
	Let the distance-generating function $\Phi:\huaX\to\mathbb R$ be a differentiable and $\sigma_\Phi$-strongly convex function (w.r.t. norm $\|\cdot\|$) on $\huaX$. The Bregman divergence $D_\Phi(x\|y)$ between $x\in\huaX$ and $y\in\huaX$ induced by $\Phi$ is defined as $D_\Phi(x\|y)=\Phi(x)-\Phi(y)-\nn\nabla\Phi(y), x-y\mm$. 
\end{defi}

  A direct consequence of Definition \ref{defiphi} is the relation between Bregman divergence and the space norm:  $D_{\Phi}(x\| y)\geq \f{\sigma_\Phi}{2}\|x-y\|^2$.
A basic result about Bregman divergence following from the definition is listed below.
\begin{lem}\label{bregman}
	(Lemma 2 in Reference \cite{xzhyx2022})	The Bregman divergence $D_\ppp(x\|y)$ satisfies the three-point identity $\nn\nabla\Phi(x)-\nabla\Phi(y),y-z\mm=D_{\Phi}(z\|x)-D_{\Phi}(z\|y)-D_{\Phi}(y\|x)$
	for all $x, y, z\in \huaX$.
\end{lem}
\begin{ass}\label{as1}
The distance-generating function $\Phi$ satisfies Definition \ref{defiphi}.	The Bregman divergence $D_{\Phi}(x\|y)$ is assumed to satisfy the separate convexity on the second variable, namely, for any $y_j\in\huaX$, $j\in\huaV$ and $\sum_{j=1}^ma_j=1$, $a_j\geq 0$, it holds that $D_{\Phi}(x\| \sum_{j=1}^ma_jy_j)\leq\sum_{j=1}^ma_jD_{\Phi}(x\|y_j).$
\end{ass}

Separate convexity is a widely adopted assumption in the literature of DMD algorithms, as seen in works such as \cite{yhhj2018}, \cite{yhhx2020}, \cite{md1}, \cite{lj2021}. This assumption holds for commonly used Bregman divergence induced by the well-known
	distance-generating functions, for example, the Bregman divergence $D_\Phi(x\|y)=\f{1}{2}\|x-y\|_2^2$ induced by the norm squared distance-generating function $\Phi(x)=\f{1}{2}\|x\|_2^2$; the Kullback-Leibler divergence $D_\Phi(x\|y)=\sum_{i=1}^n[x]_i\ln\f{[x]_i}{[y]_i}$ generated by the Gibbs entropy function $\Phi(x)=\sum_{i=1}^n[x]_i\ln [x]_i$ on the commonly used probability simplex.
	The introduction of this assumption is to accommodate the column randomness of the communication matrix to further establish the convergence of the algorithm. Compared with the non-distributed scenario, the challenge lies in how to remove this assumption in future work. However, due to the need for our distributed algorithm to analyze the communication nature of the network system, completely removing or relaxing this assumption poses significant challenges in distributed settings.

Before establishing the main results, we provide the following definition of high-probability convergence applicable to this paper (\cite{qlxc2025}, \cite{lwzw2023}, \cite{mdb2024}).
\begin{defi}
	(Convergence with High Probability) Given a vector $Y$ of ($\mbb R^m$,$\|\cdot\|$) and a sequence of random vectors, denoted by $\{X_T\}_{T=1,2,...}$, we say $X_T$ converges to $Y$ in ($\mbb R^m$,$\|\cdot\|$) in high probability if $\mbb P\{\|X_T-Y\|\leq r(T)(\log\f{1}{\delta})^\gamma\}\geq1-\delta$, for any failure probability $\delta\in(0,1)$, where $r(T)$ is positive such that $\lim_{T\rightarrow\infty}r(T)=0$ and $\gamma\geq1$.	
\end{defi}
\section{Main results and discussions}

For solving the distributed composite optimization problem \eqref{problem}, we study the following DCSMD-SW algorithm
\begin{eqnarray}
	\left\{
	\beal y_{i,t}=&\sum_{j=1}^m[W_t]_{ij}x_{j,t},\\
	x_{i,t+1}=&\arg\min_{x\in \huaX}\Big\{\left\nn \wh g_{i,t}^{SW},x\right\mm+\f{1}{\aaa_t}D_{\Phi}(x\|y_{i,t})+\psi_i(x)\Big\},
	\eeal
	\right.  \label{main_algorithm}
\end{eqnarray}
with local sub-Weibull stochastic gradient $\wh g_{i,t}^{SW}$ associated with the agent $i\in\huaV$. The algorithm framework of \eqref{problem} is essentially decentralized. After initialization with initial value $x_{i,1}$, $i\in\huaV$, for each iteration,  in the first step, each agent $i\in\huaV$ is allowed to collaborate with its instant neighbors. The agents obtain weighted average information of their neighbors' state variables to update and generate an intermediate state variable $y_{i,t}$. In the second step, this intermediate state variable is processed within a mirror descent algorithm framework with sub-Weibull stochastic gradient $\wh g_{i,t}^{SW}$, the Bregman divergence $D_\ff(\cdot\|\cdot)$ and regularization function $\psi_i$. We clarify that the DCSMD structure \eqref{main_algorithm} has already been studied firstly and primarily in \cite{yhhx2020} from a significantly different research line, although \cite{yhhx2020} considered a common regularization function technically while this paper relaxes this constraint and allows each node $i\in\huaV$ to have its own $\psi_i$
	that is accessible only to it. The core novelty of this work lies in the analytical framework for this type of algorithms under sub-Weibull randomness. We use terminology DCSMD-SW mainly to highlight the influence of sub-Weibull noise on DCSMD.

It is important to note that the local randomness here is induced by sub-Weibull random variables, in contrast significantly to the traditional light-tailed gradient randomness, considered and studied by different types of variants of the DSMD schemes in existing literature of DO (e.g. \cite{yhhj2018}, \cite{yhhx2020}, \cite{llww2018}, \cite{xzhyx2022}, \cite{fzy2024}). Moreover, the main results derived in these previous works rely on the expression of convergence in expectation.  
It makes sense to further establish a high-probability convergence theory for the DSMD schemes in a sub-Weibull setting. Through the main results presented below, we will implement it by comprehensively establishing the high-probability convergence theory of DCSMD-SW.

 The first main result indicates a general high-probability bound related to the global objective funtion $F$. The result is also the foundation of other further results on high-probability convergence rates of DCSMD-SW.
\begin{thm}\label{main_general_bdd}
	Let Assumptions \ref{subweibull_ass}-\ref{as1} hold. Let $0<\aaa_t<1$ be a non-increasing stepsize of DCSMD-SW. Then for any $\ell\in\huaV$, there holds that, for any failure probability  $0<\delta<\f{1}{10}$, with probability at least $1-10\delta$,
	\bes
	&&\sum_{t=1}^T\aaa_t\Big[F( x_{\ell,t})-F(x^*)\Big]\\ &&\leq\f{1}{d_1}\sum_{j=1}^mD_{\ff}(x^*\|x_{j,1})+\bm{\huaZ}_\huaP\left(\delta,(\aaa_t)_{t=1}^T\right)\\
	&&+\bm{\huaZ}_\huaQ\left(\delta,(\aaa_t)_{t=1}^T\right)+\bm{\huaZ}_\huaS(\delta,(\aaa_t)_{t=1}^T)+\bm{\huaZ}_\huaT(\delta,(\aaa_t)_{t=1}^T)^{1/2}
	\ees
where $d_1=\max_{i\in\huaV}\{\sqrt{D_\ff(x^*\|x_{i,1})}\}$, and $\bm{\huaZ}_\huaP(\delta,(\aaa_t)_{t=1}^T)=B_1+\Big[B_2+B_3(\log\f{2}{\delta})^\ttt+B_4(\log\f{2T}{\delta})^\ttt\Big]\sum_{t=1}^T\aaa_t^2$; $\bm{\huaZ}_\huaQ(\delta,(\aaa_t)_{t=1}^T)=B_5+\Big[B_6+B_7(\log\f{2}{\delta})^\ttt+B_8(\log\f{2T}{\delta})^\ttt\Big]\sum_{t=1}^T\aaa_t^2$; $\bm{\huaZ}_\huaS(\delta,(\aaa_t)_{t=1}^T)= B_9\left(\log\f{2T}{\delta}\right)^{\max\{0,\ttt-1\}}\log\f{1}{\delta}+ B_{10}\sum_{t=1}^T\aaa_t^2$; $\bm{\huaZ}_\huaT(\delta,(\aaa_t)_{t=1}^T)^{1/2}=B_{11}\sum_{t=1}^T\aaa_t^2+B_{12}\left(\log\f{2}{\delta}\right)^{2\ttt}\sum_{t=1}^T\aaa_t^2$,
where $\ttt\geq\f{1}{2}$ is the tail parameter and $B_1\sim B_{12}$ are absolute constants given explicitly in the proof.
	\end{thm}

In optimization, a general convergence bound is often investigated before obtaining the core convergence rate by utilizing a specific stepsize. Such result is often obtained via some intrinsic error decomposition with modular bounds deeply related to the crucial relevant quantities of the problem
 (e.g. \cite{s1}, \cite{jam2012}, \cite{yhhx2020}-\cite{llww2018}, \cite{lj2021}, \cite{xzhyx2022}, \cite{qlxc2025}, \cite{ep2024}, \cite{nl2014}). For examples, in Theorem \ref{main_general_bdd}, we know the basic convergence bound is  clearly influenced by the estimates related to Bregman divergence $D_\Phi$, the regularization functions $\psi_i$, $i\in\huaV$ (measured by $\bm{\huaZ}_\huaP$), the network communication (implicitly related to $\bm{\huaZ}_\huaP$, $\bm{\huaZ}_\huaQ$), the sub-Weibull random noise (explicitly related to $\bm{\huaZ}_\huaS$, $\bm{\huaZ}_\huaT$ and implicitly related to $\bm{\huaZ}_\huaP$, $\bm{\huaZ}_\huaQ$ due to consensus-based analysis from the network communication under sub-Weibull noise).  

For any agent $\ell\in\huaV$, consider the local ergodic sequence $\ww x_\ell^T$ with equi-weight defined by 
\bea
\ww x_\ell^T=\f{1}{T}\sum_{t=1}^T x_{\ell,t}, \ \ell\in\huaV,  \label{eqi_weight_sequence}
\eea
The next result reveals the explicit high-probability convergence rate of DCSMD-SW. 
\begin{thm}\label{main_eqi_weight}
		Let Assumptions \ref{subweibull_ass}-\ref{as1} hold. If the stepsize of DCSMD-SW satisfies $\aaa_t=1/\sqrt{t+1}$, then for any $\ell\in\huaV$, there holds that, for any $0<\delta<1$, with probability at least $1-\delta$,
		\bes
		&&F(\ww x_{\ell}^T)-F(x^*)\leq\huaO\bigg(\Big(\log\f{T}{\delta}\Big)^{2\ttt}\f{(\log T)^2}{\sqrt{T}}\\
		&&\quad +\Big(\log\f{1}{\delta}\Big)^2\Big(\log\f{T}{\delta}\Big)^{2\max\{0,\ttt-1\}}\f{1}{\sqrt{T}}\bigg), \ \ttt\geq\f{1}{2},
		\ees
\end{thm}

The influential work \cite{nl2014} and its references indicate that, studying the convergence performance of various weighted averaging sequences, particularly those with non-equal weights, has important theoretical significance and has received widespread attention and interest. In the existing work of DMD, the convergence theory for non-equal weight ergodic local sequences is rarely established, as compared with the equiweight case of $\ww x_\ell^T$. It would be interesting to broaden the existing  results of weighted averaging local sequences in DMD algorithms. Our next theorem is along this research line. We show that, for the  non-equiweight ergodic sequence $\wh x_\ell^T=\f{\sum_{t=1}^T\aaa_tx_{\ell,t}}{\sum_{t=1}^T\aaa_t}$, $\ell\in\huaV$,
	we are able to obtain the same high-probability convergence rate for DCSMD-SW.
	\begin{thm}\label{main_thm_nonequi_weight}
		Let Assumptions \ref{subweibull_ass}-\ref{as1} hold. If the stepsize of DCSMD-SW satisfies $\aaa_t=1/\sqrt{t+1}$, then for any $\ell\in\huaV$,  it holds that, for any $0<\delta<1$, with probability at least $1-\delta$,
		\bes
	&&F(\wh x_{\ell}^T)-F(x^*)\leq\huaO\bigg(\Big(\log\f{T}{\delta}\Big)^{2\ttt}\f{(\log T)^2}{\sqrt{T}}\\
	&&\quad +\Big(\log\f{1}{\delta}\Big)^2\Big(\log\f{T}{\delta}\Big)^{2\max\{0,\ttt-1\}}\f{1}{\sqrt{T}}\bigg), \ \ttt\geq\f{1}{2}.
	\ees
\end{thm}

The next result indicates that, if the time-horizon $T$ is known, via a constant strategy for stepsize selection, we are able to eliminate the influence of an order of $(\log T)^2$ in the above derived high-probability convergence rate.
	\begin{thm}\label{main_thm_known_time}
	Let Assumptions \ref{subweibull_ass}-\ref{as1} hold. Let the stepsize of DCSMD-SW be $\aaa_t=1/\sqrt{T}$, $t=1,2,...,T$. Then for any $\ell\in\huaV$, it holds that,  for any $0<\delta<1$, with probability at least $1-\delta$,
	\bes
	&&F(\ww x_{\ell}^T)-F(x^*)\leq\\
	&&\left\{
	\begin{aligned}
	&\huaO\bigg(\Big[\Big(\log\f{T}{\delta}\Big)^{2\ttt}+\Big(\log\f{1}{\delta}\Big)^{2}\Big]\f{1}{\sqrt{T}}\bigg), \  \f{1}{2}\leq\ttt<1, \\
	&\huaO\bigg(\Big(\log\f{T}{\delta}\Big)^{2\ttt}\f{1}{\sqrt{T}}\bigg), \ \ttt\geq1.
	\end{aligned}
	\right.
	\ees
\end{thm}

In this case, it is also interesting to witness that $\ww x_\ell^T=\wh x_\ell^T$, $\ell\in\huaV$ when $\aaa_t=1/\sqrt{T}$, $t=1,2,...,T$.

We review main existing works on DO under heavy-tailed noise (\cite{qlxc2025}-\cite{ylw2025}) and make a comparison on theoretical results between the current work and them. The DO problems in \cite{qlxc2025}-\cite{ylw2025} are studied under the bounded $p$th-moment condition ($p\in(1,2])$). Accordingly, they employ the clipping version of the stochastic gradient. Specifically, \cite{qlxc2025} studied the high-probability convergence of a clipped distributed dual averaging method with a high-probability rate of $\huaO(T^{\f{4\tau-1}{2}}+T^{\tau(1-p)}+\log\f{1}{\delta}\f{1}{\sqrt{T}})$ with $\tau\in(0,\f{1}{4})$ and $p\in(1,2]$. \cite{szcy2025} proposed a clipped SGD for solving DO problem while \cite{ylw2025} considered a clipped SGD for addressing online DO problem. The merit of the work of \cite{qlxc2025}-\cite{ylw2025} is that their theory is powerful to accommodate some more extreme heavy-tailed noise, as a result of utilizing bounded $p$th-moment condition. However, for  commonly encountered light-tailed (e.g. $\f{1}{2}\leq\ttt\leq1$) and heavier-tailed noises (e.g. $\ttt>1$), the current work has an obvious theoretical advantage in terms of convergence rate, even when the noise is norm sub-Gaussian or sub-Exponential. This can be witnessed since the main dominant term in \cite{qlxc2025} is  $\huaO(\max(T^{\f{4\tau-1}{2}},T^{\tau(1-p)}))$, which is slower than the current rate. Meanwhile, Remark 6 in \cite{szcy2025} indicates a nice almost sure convergence with a rate of approximately $\huaO(T^{\f{1}{2p}-\f{1}{2}})$, $p\in(1,2])$ for convex case (hence at best $\huaO(T^{-\f{1}{4}})$), and a rate of approximately $\huaO(T^{\f{1}{2(2p-1)}-\f{1}{2}})$, $p\in(1,2])$ for strongly convex case (hence at best $\huaO(T^{-\f{1}{3}})$), both of which are strictly slower than the rate in this work. Meanwhile,  stricter square-summable stepsize restriction $\sum_{t=1}^\infty\aaa_t^2<\infty$ is required. Moreover, to conduct convergence analysis,  \cite{qlxc2025}, \cite{ylw2025} required the compactness condition with explicit diameter of the domain and the upper bound of the local objective function, \cite{szcy2025} required the local objective functions are continuously differentiable and \cite{ylw2025} required the $L$-smoothness of the local objective functions and compactness of the decision set. In our DCO framework with sub-Weibull noise, these technical requirements are no longer required. By utilizing  techniques of a generalized version of Freedman-type inequality for heavier tailed martingale difference sequences (\cite{mdb2024}) together with analysis tools from basic mirror descent (e.g. \cite{nl2014}), our DCO analytical framework  allows for  high-probability convergence to hold in the presence of nonsmooth objective functions.

 If we select the distance-generating function $\Phi$ as $2$-norm square $\Phi(x)=\f{1}{2}\|x\|_2^2$ and let $\psi_i=0$, $i\in\huaV$, then the Bregman divergence becomes $D_{\Phi}(x\|y)=\f{1}{2}\|x-y\|_2^2$. Accordingly, the DSCMD-SW algorithm  recovers the crucial DSGD-SW 
\be
x_{i,t+1}=P_\huaX\Big(\sum_{j=1}^m[W_t]_{ij}x_{j,t}-\aaa_t\wh g_{i,t}^{SW}\Big), i\in\huaV, \label{DSGD}
\ee
 where $P_\huaX$ denotes the projection operator $P_\huaX(\cdot)=\arg\min_{x\in\huaX}\|\cdot-x\|_2$. 
We briefly review two representative prior  works on light-tailed DSGD \cite{rnv2010} and \cite{lwzw2023}.   \cite{rnv2010} is a seminal work which focused on convex DO, and it relies on the uniform boundedness of objective function or the compactness of decision domain. \cite{lwzw2023} is a recent work which focus on nonconvex DO related to sub-Gaussian noise, in \cite{lwzw2023}, sub-Gaussian based high-probability estimates are well developed for DSGD. In contrast,  the analysis of DSGD appears only as a byproduct of this work and in the next corollary. Although the convergence rate of  DSGD-SW under sub-Weibull randomness have not been independently investigated before, compared with the above works on light-tailed DSGD, we show that DSGD can obtain good high-probability performance under a significantly broader class of noises (sub-Weibull), based on  relaxed technical assumptions. 
\begin{cor}\label{cor1}
	Let Assumptions   \ref{subweibull_ass}-\ref{gradient_bdd} hold and $\psi_i=0$, $i\in\huaV$. Let the stepsize be $\aaa_t=1/\sqrt{t+1}$. If $\ww x_\ell^T$ defined in \eqref{eqi_weight_sequence} is  generated from DSGD-SW and let $f=\sum_{i=1}^mf_i$. Then for any $\ell\in\huaV$, it holds that, for any $0<\delta<1$,  with probability at least $1-\delta$,
	\bes
	&&f(\ww x_{\ell}^T)-f(x^*)\leq\huaO\bigg(\Big(\log\f{T}{\delta}\Big)^{2\ttt}\f{(\log T)^2}{\sqrt{T}}\\
	&&\quad +\Big(\log\f{1}{\delta}\Big)^2\Big(\log\f{T}{\delta}\Big)^{2\max\{0,\ttt-1\}}\f{1}{\sqrt{T}}\bigg), \ \ttt\geq\f{1}{2}.
	\ees
	Moreover, for the known time horizon setting, if the stepsize is selected as $\aaa_t=1/\sqrt{T}$, $t=1,2,...,T$, then for any $\ell\in\huaV$, it holds that, for any $0<\delta<1$, with probability at least $1-\delta$,
\bes
&&f(\ww x_{\ell}^T)-f(x^*)\leq\\
&&\left\{
\begin{aligned}
	&\huaO\Big(\Big[\Big(\log\f{T}{\delta}\Big)^{2\ttt}+\Big(\log\f{1}{\delta}\Big)^{2}\Big]\f{1}{\sqrt{T}}\Big), \  \f{1}{2}\leq\ttt<1, \\
	&\huaO\Big(\Big(\log\f{T}{\delta}\Big)^{2\ttt}\f{1}{\sqrt{T}}\Big), \ \ttt\geq1.
\end{aligned}
\right.
\ees
\end{cor}

As mentioned in the conclusion section of \cite{lwzw2023} as an interesting and challenging open question to achieve the convergence
	in high probability for DSGD when the gradient
	estimates have heavier tails, studying DSGD under sub-Weibull randomness is interesting and meaningful. As a direct byproduct of the high-probability convergence theory in this paper, Corollary \ref{cor1} has provided a direct answer to this open question from the perspective of sub-Weibull noise.
\begin{rmk}
It is important to highlight that, throughout the analysis process of this work, we do not need to assume that our domain is bounded with an explicit diameter information as well as an explicit uniform upper bound for $D_\Phi$, as is required in  existing work related to different DMDs; accordingly, our analytical framework is also applicable to DCSMD algorithms designed for unbounded DO or DCO problems. This stands in stark contrast to existing work on DO using DMDs \cite{yhhj2018}, \cite{yhhx2020}-\cite{fzy2024}, which are fundamentally based on conditions of bounded decision space for conducting convergence analysis. The main results in this paper are also derived without assuming any
	further smoothness of the local objective function such as
	$L$-smoothness. 
\end{rmk}
\begin{rmk}
A common theoretical challenge when considering sub-Weibull randomness in DCSMD, DSMD, and DSGD is that traditional DO probability tools such as  Azuma-Hoeffding inequalities are not directly applicable within the sub-Weibull framework. This drives the development of current theory. Additionally, as the noise tail becomes heavier (i.e. $\ttt$ increases), the iterative stability and convergence performance of the three algorithms will be challenged, and these algorithms may potentially be more sensitive to the stepsize. 
Refined stepsize strategies (e.g. introducing a suitable tuning parameter) are worth considering for future research. 	The potential deterioration of convergence of DCSMD, DSMD and DSGD  	as $\ttt$ increases aligns with our theoretical results.   There are also some potential  differences  on the influence of sub-Weibull noise among these  algorithms. The (practical) impact on DCSMD may focus on its related proximal operator (e.g. $l_1$-regularization), that may be sensitive to certain outlier inputs caused by sub-Weibull noise with $\ttt>1$. For DSMD,   a heavier tailed stochastic gradient noise may lead to  fluctuations in the dual space induced by $\Phi$ during iterations. Compared to DSGD (DSMD with $\Phi=\f{1}{2}\|\cdot\|_2^2$), when the decision space has a specific geometric structure, an appropriately chosen distance-generating function $\Phi$ in DSMD may help mitigate the effects of sub-Weibull noise in practice.
	\end{rmk}
\begin{rmk}
	 ADMM is an effective approach that can be considered as a good alternative to mirror descent, especially when addressing separable optimization  with linear constraints. It would be a challenging problem to study ADMM under sub-Weibull noise setting. The potential  difficulty lies in providing effective ADMM structures  along with corresponding new high-probability analysis tools based on a suitable approximated augmented Lagrangian framework influenced by sub-Weibull randomness. \cite{ohtg2013}-\cite{blz2021} may be useful preliminary works. By relaxing Assumption 4 in \cite{pn2021} to sub-Weibull assumption, it would be possible to study distributed stochastic gradient-tracking under sub-Weibull noise. However, designing the stepsize and other parameters to ensure algorithm convergence would be quite challenging. Additionally, variance reduction techniques have obvious advantage to reduce the variance of stochastic gradient (e.g. \cite{xkk2020}), it would be possible in the future to apply them to improve the convergence performance of DO or DCO methods under sub-Weibull noise, particularly with heavier tails that have large variance.
		\end{rmk}

\section{Preliminary lemmas and basic estimates}
In the subsequent analysis,   we denote $f=\sum_{i=1}^m f_i$ and $\psi=\sum_{i=1}^m\psi_i$. Accordingly, $F$ can be represented by $F=f+\psi$. Throughout this paper, we stipulate  the notation $\wh g_{i,0}^{SW}:=0$ for any $i\in\huaV$, $\aaa_0=0$, and the operation $\sum_{i=p}^qb_i:=0$ for any $q<p$ and any summation term $\{b_i\}$. For convenience of analysis, we also stipulate that $T\geq10$.

For agent $i\in \huaV$, denote the variable $p_{i,t}$ as
\begin{eqnarray}
	p_{i,t}=x_{i,t+1}-y_{i,t}.\label{Bregman_error}
\end{eqnarray}
The next lemma provides a basic bound for $p_{i,t}$ in terms of  sub-Weibull stochastic gradient $\wh g_{i,t}^{SW}$.  
\begin{lem}\label{Bregman_proj}
	Let Assumptions \ref{gradient_bdd}, \ref{as1} hold. For each agent $i\in\huaV$, the variable $p_{i,t}$ satisfies 
	$\|p_{i,t}\|\leq\f{1}{\sigma_\Phi}\left[\left\|\wh g_{i,t}^{SW}\right\|_*+G_\psi\right]\aaa_t$.
\end{lem}
\begin{proof}
According to  the convexity of $\psi_i$ and the first-order optimality condition, there exists $\huaS_{i,t+1}\in\partial\psi_i(x_{i,t+1})$ $s.t.$ for any $x\in\huaX$, it holds that
$\nn\aaa_t \wh g_{i,t}^{SW}+\nabla\ff(x_{i,t+1})-\nabla\ff(y_{i,t})+\aaa_t \huaS_{i,t+1},x-x_{i,t+1}\mm\geq0$. Setting $x=y_{i,t}$ in above inequality, we obtain that
$\left\nn\aaa_t \wh g_{i,t}^{SW}+\nabla\ff(x_{i,t+1})-\nabla\ff(y_{i,t})+\aaa_t \huaS_{i,t+1},y_{i,t}-x_{i,t+1}\right\mm\geq0$.
The above inequality implies that $\left\nn\aaa_t(\wh g_{i,t}^{SW}+\huaS_{i,t+1}),y_{i,t}-x_{i,t+1}\right\mm\geq\nn\nb\ff(y_{i,t})-\nb\ff(x_{i,t+1}),y_{i,t}-x_{i,t+1}\mm$.
Applying Cauchy inequality to the left hand side and $\sigma_{\ff}$-strong convexity of $\ff$ to the right hand side of above inequality, we have
$\aaa_t(\|\wh g_{i,t}^{SW}\|_*+G_{\psi})\|y_{i,t}-x_{i,t+1}\|\geq\sigma_{\ff}\|y_{i,t}-x_{i,t+1}\|^2$.
Eliminate the same term $\|y_{i,t}-x_{i,t+1}\|^2$ on both sides, we obtain the desired estimate.
\end{proof}

Next result is a basic estimate on the state deviation from the average state in the scenario of sub-Weibull randomness.
\begin{lem}\label{basic_consensus}
Under Assumptions \ref{as2}-\ref{as1}, we have
	\bes
&&\hspace{-0.7cm}\|x_{i,t}-\bar x_t\|\leq 2m\ooo\gamma^{t-2}\max_{j\in\huaV}\|x_{j,1}\|+\f{mG_\psi\ooo}{\sigma_\Phi}\sum_{s=2}^{t-1}\gamma^{t-s-1}\aaa_{s-1}\\
&&+\f{2G_\psi}{\sigma_\Phi}\aaa_{t-1}+\f{\ooo}{\sigma_\Phi}\sum_{s=2}^{t-1}\gamma^{t-s-1}\sum_{j=1}^m\left\|\wh g_{j,s-1}^{SW}\right\|_*\aaa_{s-1}\\
&&+\f{1}{m}\sum_{j=1}^m\f{1}{\sigma_\Phi}\|\wh g_{j,t-1}^{SW}\|_*\aaa_{t-1}+\f{1}{\sigma_\ff}\|\wh g_{i,t-1}^{SW}\|_*\aaa_{t-1}.
\ees
\end{lem}
\begin{proof}
For $t=1$,  $\|x_{i,1}-\bar x_1\|=\|x_{i,1}-\f{1}{m}\sum_{j=1}^mx_{j,1}\|\leq2\max_{j\in\huaV}\|x_{j,1}\|$. For $t\geq2$, equation \eqref{Bregman_error} indicates that $x_{i,t}=y_{i,t-1}+p_{i,t-1}=\sum_{j=1}^m[W_{t-1}]_{ij}x_{j,t-1}+p_{i,t-1}$. Iteration implies that $x_{i,t}=\sum_{j=1}^m[\Psi(t-1,1)]_{ij}x_{j,1}+\sum_{s=2}^{t-1}\sum_{j=1}^m[\Psi(t-1,s)]_{ij}p_{j,s-1}+p_{i,t-1}$. Taking summations on both sides of the above equality shows that $\bar x_t=\f{1}{m}\sum_{j=1}^mx_{j,1}+\f{1}{m}\sum_{s=2}^{t-1}\sum_{j=1}^mp_{j,s-1}+\f{1}{m}\sum_{j=1}^mp_{j,t-1}$. Subtraction between the above two equations and taking norms yields $\|x_{i,t}-\bar x_t\|\leq\sum_{j=1}^m\left|[\Psi(t-1,1)]_{ij}-\f{1}{m}\right|\|x_{j,1}\|+\sum_{s=2}^{t-1}\sum_{j=1}^m\left|[\Psi(t-1,s)]_{ij}-\f{1}{m}\right|\|p_{j,s-1}\|+\|p_{i,t-1}\|+\f{1}{m}\sum_{j=1}^m\|p_{j,t-1}\|$. 
	According to Lemma \ref{network_nedic_lem}, we have $\|x_{i,t}-\bar x_t\|\leq\ooo\gamma^{t-2}\sum_{j=1}^m\|x_{j,1}\|+\ooo\sum_{s=2}^{t-1}\gamma^{t-s-1}\sum_{j=1}^m\|p_{j,s-1}\|+\|p_{i,t-1}\|+\f{1}{m}\sum_{j=1}^m\|p_{j,t-1}\|$.	Applying Lemma \ref{Bregman_proj} to the above inequality, we obtain the desired estimate.
\end{proof}

The next result is basic and the foundation of further convergence results. Several core terms that need to be estimated in this work are potentially contained within.
\begin{lem}\label{basic_MD_ine}
Under Assumptions \ref{gradient_bdd}, \ref{as1}, there holds
\bea
\nono&&\aaa_t\big[\psi_i(x_{i,t+1})-\psi_i(x^*)\big]+\left\nn\aaa_t\wh g_{i,t}^{SW},y_{i,t}-x^*\right\mm\\
&&\leq D_{\ff}(x^*\|y_{i,t})-D_{\ff}(x^*\|x_{i,t+1}) +\f{\aaa_t^2}{2\sigma_{\ff}}\|\wh g_{i,t}^{SW}\|_*^2.
\eea	
\end{lem}
\begin{proof}
	Setting $x=x^*$ in the first inequality in the proof of Lemma \ref{Bregman_proj}, and rearranging terms, we have
	\bea
	\nono&&\left\nn\aaa_t\wh g_{i,t}^{SW},x_{i,t+1}-x^*\right\mm   \\
	\nono&&\leq\nn\nb\ff(y_{i,t})-\nb\ff(x_{i,t+1}),x_{i,t+1}-x^*\mm\\
	\nono&&\quad+\aaa_t\nn \huaS_{i,t+1},x^*-x_{i,t+1}\mm  \\
	\nono&&\leq D_{\ff}(x^*\|y_{i,t})-D_{\ff}(x^*\|x_{i,t+1})-D_{\ff}(x_{i,t+1}\|y_{i,t})\\
	\nono&&\ \ \ +\psi_i(x^*)-\psi_i(x_{i,t+1})\\
	\nono&&\leq D_{\ff}(x^*\|y_{i,t})-D_{\ff}(x^*\|x_{i,t+1})-\f{\sigma_{\ff}}{2}\|x_{i,t+1}-y_{i,t}\|^2\\
	&&\ \ \ +\psi_i(x^*)-\psi_i(x_{i,t+1}),\label{nei1}
	\eea
	in which the second inequality follows from the three-point inequality in Lemma \ref{bregman} and the third inequality follows from the definition of $D_{\ff}(\cdot,\cdot)$ and $\sigma_{\ff}$-strong convexity of $\ff$. Also,
	\bea
	\nono&&\left\nn\aaa_t\wh g_{i,t}^{SW},x_{i,t+1}-x^*\right\mm\\
	\nono&&=\left\nn\aaa_t\wh g_{i,t}^{SW},x_{i,t+1}-y_{i,t}\right\mm+\left\nn\aaa_t\wh g_{i,t}^{SW},y_{i,t}-x^*\right\mm\\
	\nono&& \geq-\f{\aaa_t^2}{2\sigma_{\ff}}\left\|\wh g_{i,t}^{SW}\right\|_*^2-\f{\sigma_{\ff}}{2}\|x_{i,t+1}-y_{i,t}\|^2\\
	&&\quad +\left\nn\aaa_t\wh g_{i,t}^{SW},y_{i,t}-x^*\right\mm. \label{nei2}
	\eea
	Combining \eqref{nei1} and \eqref{nei2}, we obtain the desired result.
	\end{proof}

Based on Lemma \ref{basic_MD_ine}, note that the stochastic gradient $\wh g_{i,t}^{SW}$ satisfies $\wh g_{i,t}^{SW}=g_{i,t}-\xi_{i,t}$. If we denote 
\bea
&&\hspace{-1cm}\ww d_t=\max_{s\in[t]}\left\{\rho,d_{s}\right\},\label{piao_d_def}\\
&&\hspace{-1cm}d_t=\max_{i\in\huaV}\left\{\sqrt{D_\ff(x^*\|x_{i,t})}\right\},
\eea
with $\rho>0$ a constant that will be determined in the subsequent analysis, then we know, for any $S\in[T]$, it holds that,
\bea
\huaP_{S}+\huaQ_{S}\leq\huaR_{S}+\huaS_{S}+\huaT_{S},  \label{orignal_est}
\eea
where $\huaP_S=\sum_{t=1}^S\f{\aaa_t}{\ww d_t}\sum_{i=1}^m[\psi_i(x_{i,t+1})-\psi_i(x^*)]$, $\huaQ_S=\sum_{t=1}^S\f{\aaa_t}{\ww d_t}\sum_{i=1}^m\nn  g_{i,t},y_{i,t}-x^*\mm$, $\huaR_S=\sum_{t=1}^S\f{1}{\ww d_t}\sum_{i=1}^m[D_{\ff}(x^*\|y_{i,t})-D_{\ff}(x^*\|x_{i,t+1})]$, $\huaS_S=\sum_{t=1}^S\f{\aaa_t}{\ww  d_t}\sum_{i=1}^m\nn  \xi_{i,t},y_{i,t}-x^*\mm$, $\huaT_S=\sum_{t=1}^S\f{\aaa_t^2}{2\sigma_\Phi \ww d_t}\sum_{i=1}^m\|\wh g_{i,t}^{SW}\|_*^2$.
Our upcoming efforts will focus on deriving the high-probability bounds related to these five terms. In the following, we denote the random variables
\bea
\ww \xi_{i,t}=\|\xi_{i,t}\|_*, \label{xi_def1}\\
\wh \xi_{i,t}=\|\xi_{i,t}\|_*^2. \label{xi_def2}
\eea
We will need the following result for arbitrary moments of sub-Weibull random variables.
\begin{lem}\label{moment_subweibull}
(Lemma 6 in Reference \cite{mdb2024}) If $X$ is sub-Weibull ($\ttt$,$\kkk$) random variable. Then, for any $p>0$, it holds that 
	$\mathbb E[|X|^p]\leq2\Gamma(\ttt p+1)\kkk^p$.
\end{lem}

The next lemma indicates that the summation of sub-Weibull random variables is a sub-Weibull random variable.
\begin{lem}\label{summation_SubW}
	Let $X_i$, $i=1,2,...,m$ be sub-Weibull ($\ttt$,$\kkk_i$) random variables. Then, their summation satisfies $$\sum_{i=1}^mX_i\sim SubW\left(\ttt,K_\ttt\sum_{i=1}^m\kkk_i\right),$$ where $K_\ttt=m^\ttt$ for $\ttt>1$ and $K_\ttt=1$ for $0<\ttt\leq1$.
\end{lem}
\begin{proof}
	For $\ttt>1$, setting $\kkk=m^{\ttt}\sum_{i=1}^m\kkk_i$, it holds that
	\bes
	&&\hspace{-0.5cm}\mbb E\exp\left(\f{\left|\sum_{i=1}^mX_i\right|^{1/\ttt}}{\kkk^{1/\ttt}}\right)\leq\mbb E\exp\left(\f{\sum_{i=1}^m|X_i|^{1/\ttt}}{m(\sum_{i=1}^m\kkk_i)^{1/\ttt}}\right)\\
	&&\hspace{-0.5cm}\leq\mbb E\left\{\prod_{i=1}^m\exp\left(\f{|X_i|^{1/\ttt}}{m\kkk_i^{1/\ttt}}\right)\right\}\leq\f{1}{m}\sum_{i=1}^m\mbb E\exp\left(\f{|X_i|^{1/\ttt}}{\kkk_i^{1/\ttt}}\right)\leq2,
	\ees
	which shows $\sum_{i=1}^mX_i\sim SubW(\ttt,m^\ttt\sum_{i=1}^m\kkk_i)$  when $\ttt>1$.
	
	For $0<\ttt\leq1$, as a result of the convexity of the function $\exp(x^{1/\ttt})$, we have
	\bes
	\exp\left\{\left(\f{|\sum_{i=1}^mX_i|}{\sum_{i=1}^m\kkk_i}\right)^{1/\ttt}\right\}\leq\sum_{i=1}^m\f{\kkk_i}{\sum_{j=1}^m\kkk_j}\exp\left\{\left(\f{|X_i|}{\kkk_i}\right)^{1/\ttt}\right\}.
	\ees
	After taking expectations on both sides, we obtain $\sum_{i=1}^mX_i\sim SubW(\ttt,\sum_{i=1}^m\kkk_i)$  when $0<\ttt\leq1$.
	\end{proof}

The next basic lemma is about the concentration of several sub-Weibull variables. 
\begin{lem}\label{subweibull_concentration}
(Lemma A.3 in Reference \cite{ll2022})	Suppose $X_i\sim SubW(\ttt,\kkk_i)$, $i=1,2,...,k$. Then for any $r>0$, it holds that
	\bes
\mbb P\left\{\left|\sum_{i=1}^kX_i\right|\geq r\right\}\leq2\exp\left\{-\left(\f{r}{v_\ttt\sum_{i=1}^k\kkk_i}\right)^{1/\ttt}\right\}
	\ees
	with $v_\ttt=(4e)^{\ttt}$ for $0<\ttt\leq1$ and $v_\ttt=2(2e\ttt)^\ttt$ for $\ttt\geq1$.
\end{lem}
\section{Deriving core high-probability bounds}
\subsection{High probability estimates for $\huaP_S$}
This part aims at providing a high-probability lower bound for $\huaP_S$ defined above. $\huaP_S$ involves the regularization functions of the main problem \eqref{problem} and DCSMD-SW.  The main results of this subsection reveal the core impact of sub-Weibull noises on the information communications, consensus and related disagreements between nodes, as well as their core influence on the estimates related to local regularization functions.
We start with the following decomposition, for any $\ell\in\huaV$, 
\bea
\nono&&\f{\aaa_t}{\ww d_t}[\psi_i(x_{i,t+1})-\psi_i(x^*)]\\
\nono&&=\f{\aaa_t}{\ww d_t}[\psi_i(x_{i,t+1})-\psi_i(y_{i,t})]+\f{\aaa_t}{\ww d_t}[\psi_i(y_{i,t})-\psi_i(x_{\ell,t})]\\
\nono&&+\f{\aaa_t}{\ww d_t}[\psi_i(x_{\ell,t})-\psi_i(x^*)]\\
\nono&&\geq -G_{\psi}\f{\aaa_t}{\ww d_t}\|x_{i,t+1}-y_{i,t}\|-G_{\psi}\f{\aaa_t}{\ww d_t}\|y_{i,t}-x_{\ell,t}\|\\
&&+\f{\aaa_t}{\ww d_t}[\psi_i(x_{\ell,t})-\psi_i(x^*)].  
\eea
Hence, we have, for any $S\in[T]$, there holds
\be
\huaP_S\geq-\ww \huaP_T+\sum_{t=1}^S\f{\aaa_t}{\ww d_t}\sum_{i=1}^m[\psi_i(x_{\ell,t})-\psi_i(x^*)],   \label{bdd_supportP}
\ee
where
\bea
\beal
\ww \huaP_T=&G_{\psi}\sum_{t=1}^T\f{\aaa_t}{\ww d_t}\sum_{i=1}^m\|x_{i,t+1}-y_{i,t}\|\\
&+G_{\psi}\sum_{t=1}^T\f{\aaa_t}{\ww d_t}\sum_{i=1}^m\|y_{i,t}-x_{\ell,t}\|.\label{RT_start}
\eeal
\eea
For the first term on the right hand side of \eqref{RT_start}, we have the following high-probability bound.

\begin{pro}\label{RT_first_term_est}
Under Assumptions \ref{subweibull_ass}, \ref{gradient_bdd}, \ref{as1}, it holds that, for any $\delta_1\in(0,1)$, with probability at least $1-\delta_1$,
	\bes
	&&G_\psi\sum_{t=1}^T\sum_{i=1}^m\f{\aaa_t}{\ww d_t}\|x_{i,t+1}-y_{i,t}\|\\
	&&\leq C_1\sum_{t=1}^T\aaa_t^2+C_2\left(\log\f{2}{\delta_1}\right)^\ttt\sum_{t=1}^T \aaa_t^2,
	\ees
	where $C_1=\f{mG_\psi}{\sigma_\Phi d_1}[G+G_\psi]$ and $C_2=\f{G_\psi}{\sigma_\Phi d_1}v_\ttt  m^{\ttt+1}\kkk$.
\end{pro}
\begin{proof}
	Note that the following estimates hold, 
\bes
&&\f{\aaa_t}{\ww d_t}\|x_{i,t+1}-y_{i,t}\|\leq\f{1}{\sigma_\Phi}\big[\|\wh g_{i,t}^{SW}\|_*+G_\psi\big]\f{\aaa_t^2}{\ww d_t}\\
&&\leq\f{1}{\sigma_\Phi}\left(\| g_{i,t}\|_*+\| \xi_{i,t}\|_*\right)\f{\aaa_t^2}{\ww d_t}+\f{G_\psi}{\sigma_\ff}\f{\aaa_t^2}{\ww d_t}\\
&&\leq\f{1}{\sigma_\Phi}(G+G_\psi)\f{\aaa_t^2}{\ww d_t}+\f{1}{\sigma_\ff}\f{\aaa_t^2}{\ww d_t}\ww \xi_{i,t},
\ees
where we have used Assumption \ref{gradient_bdd} and Lemma \ref{Bregman_proj}. After taking summation on both sides for iteration $t$  and agent number $i$, we have, 
\bes
&&\sum_{t=1}^T\sum_{i=1}^m\f{\aaa_t}{\ww d_t}\|x_{i,t+1}-y_{i,t}\|\\
&&\leq \f{m(G+G_\psi)}{\sigma_\Phi}\sum_{t=1}^T\f{\aaa_t^2}{ d_1}+\f{1}{\sss_\ff}\sum_{t=1}^T\f{\aaa_t^2}{ d_1}\sum_{i=1}^m\ww \xi_{i,t}.
\ees
According to Assumption \ref{subweibull_ass}, Since $\ww\xi_{i,t}\sim SubW(\ttt,\kkk)$ conditioned on $\huaF_{t-1}$, then we have that $\aaa_t^2\ww \xi_{i,t}\sim SubW(\ttt,\kkk\aaa_t^2)$ conditioned on $\huaF_{t-1}$ for $t\in[T]$, and $\aaa_t^2\ww \xi_{i,t}\sim SubW(\ttt,\kkk\aaa_t^2)$ after noting that 
$\mbb E[\exp\{(|\aaa_t^2\ww \xi_{i,t}|/\kkk\aaa_t^2)^{1/\ttt}\}]=\mbb E[\mbb E[\exp\{(|\aaa_t^2\ww \xi_{i,t}|/\kkk\aaa_t^2)^{1/\ttt}\}|\huaF_{t-1}]]$.
 Then it follows from Lemma \ref{summation_SubW} that $\aaa_t^2\sum_{i=1}^m\ww \xi_{i,t}\sim SubW\left(\ttt,m^{\ttt+1}\kkk\aaa_t^2\right)$. Then by utilizing Lemma \ref{subweibull_concentration},
 we have, for any $r>0$,
 $\mbb P\left\{\left|\sum_{t=1}^T\aaa_t^2\sum_{i=1}^m\ww \xi_{i,t}\right|\geq r\right\}\leq2\exp\left\{-\left(\f{r}{v_\ttt  m^{\ttt+1}\kkk\sum_{t=1}^T\aaa_t^2}\right)^{1/\ttt}\right\}$,
  with $v_\ttt$ defined in Lemma \ref{subweibull_concentration}. Therefore,  with probability at least $1-\delta_1$,
\bes
\sum_{t=1}^T\aaa_t^2\sum_{i=1}^m\ww \xi_{i,t}
\leq v_\ttt  m^{\ttt+1}\kkk\left(\log \f{2}{\delta_1}\right)^\ttt\sum_{t=1}^T\aaa_t^2.
\ees
Accordingly, we have, with probability at least $1-\delta_1$, the desired result holds. We complete the proof.
\end{proof}

For the second term on the right hand side of \eqref{RT_start}, we have the following result on its high-probability bound.
\begin{pro}\label{RT_second_term_est}
Under Assumptions \ref{subweibull_ass}-\ref{as1}, Let $0<\aaa_t<1$ be a non-increasing stepsize, then	for any $\ell\in\huaV$, we have, for any $\delta_i\in(0,1)$, $i=1,2,3,4$, satisfying $\sum_{i=1}^4\delta_i<1$,  with probability at least $1-\delta_2-\delta_3-\delta_4$,
	\bes
	&&\hspace{-0.5cm} G_{\psi}\sum_{t=1}^T\f{\aaa_t}{\ww d_t}\sum_{i=1}^m\|x_{\ell,t}-y_{i,t}\|\leq C_3+C_4\sum_{t=1}^T\aaa_t^2\\
	&&\hspace{-0.5cm}+\Big[C_5'\left(\log\f{2T}{\delta_2}\right)^\ttt+C_6'\left(\log\f{2}{\delta_3}\right)^\ttt+C_7'\left(\log\f{2}{\delta_4}\right)^\ttt\Big]\sum_{t=1}^T\aaa_t^2
	\ees
	with $C_3=\f{4m^2\ooo\gamma^{-1} G_\psi}{(1-\gamma)d_1}\max_{j\in\huaV}\|x_{j,1}\|$, $C_4=\f{G_\psi}{d_1}\{\f{2G_\psi}{\sss_\ff}(\f{m^2\ooo}{1-\gamma}+2m)+(\f{2m^2\ooo}{\sss_\ff(1-\gamma)}+\f{4m}{\sss_\ff})G\}$, $C_5'=\f{2m\ooo G_\psi}{\sss_\ff d_1}v_\ttt  m^{\ttt+1}\kkk\f{1}{1-\gamma}$, $C_6'=\f{3G_\psi}{\sss_\ff d_1}v_\ttt  m^{\ttt+1}\kkk$ and $C_7'=\f{mG_\psi}{\sss_\ff d_1}v_\ttt \kkk$.
	\end{pro}
\begin{proof}
	For any $\ell\in\huaV$, according to the double stochasticity of the matrix $W_t$ in Assumption \ref{as2} and the convexity of the norm $\|\cdot\|$, $\|x_{\ell,t}-y_{i,t}\|=\big\|x_{\ell,t}-\sum_{j=1}^m[W_t]_{ij}x_{j,t}\big\|\leq \sum_{j=1}^m[W_t]_{ij} \|x_{\ell,t}-x_{j,t}\|\leq\sum_{j=1}^m[W_t]_{ij} \left\|x_{j,t}-\bar x_t\right\|+\|x_{\ell,t}-\bar x_t\|$.
Note that, $\aaa_t\leq \aaa_s$ for any $1\leq s\leq t-1$. It follows from Lemma \ref{basic_consensus} that, for any   $\ell\in\huaV$,
\bea
\nono&&\sum_{t=1}^T\f{\aaa_t}{\ww d_t}\sum_{i=1}^m\|x_{\ell,t}-y_{i,t}\|\\
\nono&&\leq\sum_{t=1}^T\sum_{j=1}^m\f{\aaa_t}{\ww d_t}\|x_{j,t}-\bar x_t\|+m\sum_{t=1}^T\f{\aaa_t}{\ww d_t}\|\bar x_t-x_{\ell,t}\|\\
\nono&&\leq \f{4m^2\gamma^{-1}\ooo}{(1-\gamma)d_1}\max_{j\in\huaV}\|x_{j,1}\|+\f{2G_\psi}{\sss_\ff d_1}\left(\f{m^2\ooo}{1-\gamma}+2m\right)\sum_{t=1}^T\aaa_t^2\\
\nono&&\quad+\f{2m\ooo}{\sss_\ff d_1}\sum_{t=1}^T\sum_{s=2}^{t-1}\gamma^{t-s-1}\sum_{j=1}^m\left\|\wh g_{j,s-1}^{SW}\right\|_*\aaa_{s-1}^2\\
\nono&&\quad+\f{3}{\sss_\ff d_1}\sum_{t=1}^T\sum_{j=1}^m\left\|\wh g_{j,t-1}^{SW}\right\|_*\aaa_{t-1}^2\\
&&\quad+\f{m}{\sss_\ff d_1}\sum_{t=1}^T\left\|\wh g_{\ell,t-1}^{SW}\right\|_*\aaa_{t-1}^2. \label{xlt_yit}
\eea	
Note that $\sum_{j=1}^m\left\|\wh g_{j,t-1}^{SW}\right\|_*\aaa_{t-1}^2\leq\sum_{j=1}^m[
\|g_{j,t-1}\|_*+\|\xi_{j,t-1}\|_*]\aaa_{t-1}^2$.
According to Assumption \ref{gradient_bdd}, we have, $\sum_{j=1}^m\|\wh g_{j,t-1}^{SW}\|_*\aaa_{t-1}^2\leq mG\aaa_{t-1}^2+\sum_{j=1}^m\ww \xi_{j,t-1}\aaa_{t-1}^2.$
In a similar way, we have $\sum_{j=1}^m \|\wh g_{j,s-1}^{SW}\|_*\aaa_{s-1}^2\leq mG\aaa_{s-1}^2+\sum_{j=1}^m\ww \xi_{j,s-1}\aaa_{s-1}^2,$
and for any $\ell\in\huaV$, $\|\wh g_{\ell,t-1}^{SW}\|_*\aaa_{t-1}^2\leq G\aaa_{t-1}^2+\ww \xi_{\ell,t-1}\aaa_{t-1}^2.$
Based on the above estimates, we arrive at
\bes
&&G_{\psi}\sum_{t=1}^T\f{\aaa_t}{\ww d_t}\sum_{i=1}^m\|x_{\ell,t}-y_{i,t}\|\\
&&\leq C_3+C_4\sum_{t=1}^T\aaa_t^2+C_5\sum_{t=1}^T\sum_{s=2}^{t-1}\gamma^{t-s-1}\sum_{j=1}^m\ww \xi_{j,s-1}\aaa_{s-1}^2\\
&&\quad+C_6\sum_{t=1}^T\aaa_{t-1}^2\sum_{j=1}^m\ww \xi_{j,t-1}+C_7\sum_{t=1}^T\ww \xi_{\ell,t-1}\aaa_{t-1}^2
\ees
with $C_3=\f{4m^2\ooo\gamma^{-1} G_\psi}{(1-\gamma)d_1}\max_{j\in\huaV}\|x_{j,1}\|$, $C_4=\f{G_\psi}{d_1}\{\f{2G_\psi}{\sss_\ff}(\f{m^2\ooo}{1-\gamma}+2m)+(\f{2m^2\ooo}{\sss_\ff(1-\gamma)}+\f{4m}{\sss_\ff})G\}$, $C_5=\f{2m\ooo G_\psi}{\sss_\ff d_1}$, $C_6=\f{3 G_\psi}{\sss_\ff d_1}$, $C_7=\f{m G_\psi}{\sss_\ff d_1}$. The fact that
 $\ww \xi_{i,t}\sim SubW(\ttt,\kkk)$, $i\in\huaV$ implies  $$\gamma^{t-s-1}\aaa_{s-1}^2\sum_{j=1}^m\ww \xi_{j,s-1}\sim  SubW(\ttt,\gamma^{t-s-1}\aaa_{s-1}^2 m^{\ttt+1}\kkk).$$ Then for any $t\in[T]$, according to Lemma \ref{subweibull_concentration}, it holds that, for any $r>0$,
 \bes
 &&\mbb P\left\{\left|\sum_{s=2}^{t-1}\gamma^{t-s-1}\aaa_{s-1}^2\sum_{j=1}^m\ww \xi_{j,s-1}\right|\geq r\right\}\\
 &&\leq 2\exp\left\{-\left(\f{r}{v_\ttt  m^{\ttt+1}\kkk \sum_{s=2}^{t-1}\gamma^{t-s-1}\aaa_{s-1}^2}\right)^{1/\ttt}\right\}.
 \ees
 Hence, for any $t\in[T]$, with probability at least $1-\delta_2$,
 \bes
 &&\sum_{s=2}^{t-1}\gamma^{t-s-1}\aaa_{s-1}^2\sum_{j=1}^m\ww \xi_{j,s-1}\\
 &&\leq\left(\log\f{2}{\delta_2}\right)^\ttt v_\ttt   m^{\ttt+1}\kkk \sum_{s=2}^{t-1}\gamma^{t-s-1}\aaa_{s-1}^2.
 \ees
 Accordingly, there holds that, with probability at least $1-T\delta_2$,
  \bes
&&\sum_{t=1}^T \sum_{s=2}^{t-1}\gamma^{t-s-1}\aaa_{s-1}^2\sum_{j=1}^m\ww \xi_{j,s-1}\\
&&\leq\left(\log\f{2}{\delta_2}\right)^\ttt v_\ttt  m^{\ttt+1}\kkk\sum_{t=1}^T \sum_{s=2}^{t-1}\gamma^{t-s-1}\aaa_{s-1}^2\\
&&\leq\left(\log\f{2}{\delta_2}\right)^\ttt \f{v_\ttt  m^{\ttt+1}\kkk}{1-\gamma}\sum_{t=1}^{T}\aaa_{t}^2.
 \ees
 On the other hand, following similar analysis with  above, the facts $\sum_{i=1}^m\ww \xi_{i,t-1}\sim SubW(\ttt,  m^{\ttt+1}\kkk)$ and $\ww \xi_{\ell,t-1}\aaa_{t-1}^2\sim SubW(\ttt,\kkk\aaa_{t-1}^2)$ together with Lemma \ref{subweibull_concentration}, we have, with probability at least $1-\delta_3$,
 \bes
 \sum_{t=1}^T\sum_{i=1}^m\ww \xi_{i,t-1}\aaa_{t-1}^2\leq \left(\log\f{2}{\delta_3}\right)^\ttt v_\ttt m^{\ttt+1}\kkk\sum_{t=1}^T\aaa_{t-1}^2
 \ees
 and with probability at least $1-\delta_4$,
 \bes
 \sum_{t=1}^T\ww \xi_{\ell,t-1}\aaa_{t-1}^2\leq\left(\log\f{2}{\delta_4}\right)^\ttt v_\ttt \kkk\sum_{t=1}^T\aaa_{t-1}^2.
 \ees
 Combining the above high-probability bounds, we have,  with probability at least $1-\delta_2-\delta_3-\delta_4$,
 \bes
&&\hspace{-0.5cm} G_{\psi}\sum_{t=1}^T\f{\aaa_t}{\ww d_t}\sum_{i=1}^m\|x_{\ell,t}-y_{i,t}\|\leq C_3+C_4\sum_{t=1}^T\aaa_t^2\\
 &&\hspace{-0.5cm}+\Big[C_5'\left(\log\f{2T}{\delta_2}\right)^\ttt+C_6'\left(\log\f{2}{\delta_3}\right)^\ttt+C_7'\left(\log\f{2}{\delta_4}\right)^\ttt\Big]\sum_{t=1}^T\aaa_t^2
 \ees
 with $C_5'=\f{2m\ooo G_{\psi}}{\sss_\ff d_1}v_\ttt   m^{\ttt+1}\kkk\f{1}{1-\gamma}$, $C_6'=\f{3G_{\psi}}{\sss_\ff d_1}v_\ttt   m^{\ttt+1}\kkk$ and $C_7'=\f{mG_{\psi}}{\sss_\ff d_1}v_\ttt  \kkk$.
	\end{proof}

It is easy to observe that the influence of the network topology appears from the estimate in Proposition \ref{RT_second_term_est}. Quipped with the above two propositions, we arrive at the following high-probability result for $\huaP_T$.
\begin{pro}\label{huaP_est_pro}
Let  Assumptions \ref{subweibull_ass}-\ref{as1} hold. Let $0<\aaa_t<1$ be a non-increasing stepsize, then  for any $\ell\in\huaV$ and $S\in[T]$, there holds, for any $\delta_i\in(0,1)$, $i=1,2,3,4$, with $\sum_{i=1}^4\delta_i<1$,	with probability at least $1-\sum_{i=1}^4\delta_i$,
\bes
\huaP_S\geq\sum_{t=1}^S\f{\aaa_t}{\ww d_t}\Big[\psi(x_{\ell,t})-\psi(x^*)\Big]-\huaZ_\huaP(\delta_1,\delta_2,\delta_3,\delta_4,(\aaa_t)_{t=1}^T),
\ees
with 
\bes
&&\hspace{-0.7cm}\huaZ_\huaP\left(\delta_1,\delta_2,\delta_3,\delta_4,(\aaa_t)_{t=1}^T\right)\\
&&\hspace{-0.7cm}:=C_1\sum_{t=1}^T\aaa_t^2+C_2\left(\log\f{2}{\delta_1}\right)^\ttt\sum_{t=1}^T\aaa_t^2+C_3+C_4\sum_{t=1}^T\aaa_t^2\\
&&\hspace{-0.7cm}+\Big[C_5'\left(\log\f{2T}{\delta_2}\right)^\ttt+C_6'\left(\log\f{2}{\delta_3}\right)^\ttt+C_7'\left(\log\f{2}{\delta_4}\right)^\ttt\Big]\sum_{t=1}^T\aaa_t^2.
\ees
\end{pro}
\begin{proof}
	Combining Proposition \ref{RT_first_term_est} and Proposition \ref{RT_second_term_est}, we have, for the $\ww \huaP_T$ defined in \eqref{RT_start}, it holds that, with probability at least $1-\sum_{i=1}^4\delta_i$,
	\bea
	\ww\huaP_T\leq \huaZ_\huaP\left(\delta_1,\delta_2,\delta_3,\delta_4,(\aaa_t)_{t=1}^T\right).   \label{Ppiao_est}
	\eea
	 After noting that  $\psi=\sum_{i=1}^m\psi_i$, we obtain the desired result.
	\end{proof}
\subsection{High probability estimates for $\huaQ_S$}
The next result provides a basic high-probability estimate for $\huaQ_\huaS$.
\begin{pro}\label{huaQ_est_pro}
Let Assumptions \ref{subweibull_ass}-\ref{as1} hold. Then for any $\ell\in\huaV$, $S\in[T]$, and we have, for any $\delta_i\in(0,1)$, $i=5,6,7$ with $\sum_{i=5}^7\delta_i<1$, with probability at least $1-\sum_{i=5}^7\delta_i$,
\bes
&&\huaQ_S\geq\sum_{t=1}^S\f{\aaa_t}{\ww d_t}\big[f(x_{\ell,t})-f(x^*)\big]-\huaZ_\huaQ(\delta_5,\delta_6,\delta_7,(\aaa_t)_{t=1}^T),
\ees
with
\bes
&&\huaZ_\huaQ\left(\delta_5,\delta_6,\delta_7,(\aaa_t)_{t=1}^T\right):=\ww C_3+\ww C_4\sum_{t=1}^T\aaa_t^2\\
&&+\Big[\ww C_5\left(\log\f{2T}{\delta_5}\right)^\ttt+\ww C_6\left(\log\f{2}{\delta_6}\right)^\ttt+\ww C_7\left(\log\f{2}{\delta_7}\right)^\ttt\Big]\sum_{t=1}^T\aaa_t^2
\ees
in which $\ww C_i=\f{3G}{G_\psi}C_i$, $i=3,4$ and $\ww C_j=\f{3G}{G_\psi}C_j'$, $j=5,6,7$ with previously defined $C_3$, $C_4$, $C_5'$, $C_6'$, $C_7'$.
\end{pro}
\begin{proof}
	Start from the following decomposition:
\bes
\beal
\f{\aaa_t}{\ww d_t}\nn g_{i,t},y_{i,t}-x^*\mm=\f{\aaa_t}{\ww d_t}\nn g_{i,t},y_{i,t}-x_{i,t}\mm+\f{\aaa_t}{\ww d_t}\nn g_{i,t},x_{i,t}-x^*\mm.
\eeal
\ees
The first term satisfies that
\bes
\beal
\f{\aaa_t}{\ww d_t}\nn g_{i,t},y_{i,t}-x_{i,t}\mm\geq&-\f{\aaa_t}{\ww d_t}\|g_{i,t}\|_*\|y_{i,t}-x_{i,t}\|\\
\geq&-G\f{\aaa_t}{\ww d_t}\|y_{i,t}-x_{i,t}\|.
\eeal
\ees
The second term satisfies that, for any $\ell\in\huaV$,
\bes
&&\f{\aaa_t}{\ww d_t}\nn g_{i,t},x_{i,t}-x^*\mm\geq\f{\aaa_t}{\ww d_t}\big[f_i(x_{i,t})-f_i(x^*)\big]\\
&&=\f{\aaa_t}{\ww d_t}\big[f_i(x_{i,t})-f_i(x_{\ell,t})\big]+\f{\aaa_t}{\ww d_t}\big[f_i(x_{\ell,t})-f_i(x^*)\big]\\
&&\geq-G\f{\aaa_t}{\ww d_t}\|x_{i,t}-x_{\ell,t}\|+\f{\aaa_t}{\ww d_t}\big[f_i(x_{\ell,t})-f_i(x^*)\big].
\ees
Combining the above two estimates, we know, for any $\ell\in\huaV$,
\bes
\beal
\f{\aaa_t}{\ww d_t}\nn g_{i,t},y_{i,t}-x^*\mm\geq&-G\f{\aaa_t}{\ww d_t}\Big[\|y_{i,t}-x_{i,t}\|+\|x_{i,t}-\bar x_t\|\\
&+\|\bar x_t-x_{\ell,t}\|\Big]+\f{\aaa_t}{\ww d_t}\big[f_i(x_{\ell,t})-f_i(x^*)\big].
\eeal
\ees
Then after taking summations from $t=1$ to $t=T$ and $i=1$ to $i=m$ on both sides of the above inequality and noticing the fact that $f=\sum_{i=1}^mf_i$, we have, for any $\ell\in\huaV$, 
$S\in[T]$,
\be
\huaQ_S\geq-\ww \huaQ_T+\sum_{t=1}^S\f{\aaa_t}{\ww d_t}\big[f(x_{\ell,t})-f(x^*)\big], \label{bdd_supportQ}
\ee
with 
\bea
\beal
\ww \huaQ_T=&G\sum_{t=1}^T\f{\aaa_t}{\ww d_t}\sum_{i=1}^m\|y_{i,t}-x_{i,t}\|\\
&+G\sum_{t=1}^T\f{\aaa_t}{\ww d_t}\sum_{i=1}^m\Big[\|x_{i,t}-\bar x_t\|+\|\bar x_t-x_{\ell,t}\|\Big]. \label{wwhuaQ_def}
\eeal
\eea
Since $\|y_{i,t}-x_{i,t}\|=\|\sum_{j=1}^m[W_t]_{ij}x_{j,t}-x_{i,t}\|\leq\sum_{j=1}^m[W_t]_{ij}\|x_{j,t}-\bar x_t\|+\|\bar x_t-x_{i,t}\|$. Then after taking summation, we have $\sum_{t=1}^T\f{\aaa_t}{\ww d_t}\sum_{i=1}^m\|y_{i,t}-x_{i,t}\|\leq2\sum_{t=1}^T\f{\aaa_t}{\ww d_t}\sum_{j=1}^m\|x_{j,t}-\bar x_t\|$.
Hence we have
\bes
\beal
\ww \huaQ_T\leq3G\sum_{t=1}^T\f{\aaa_t}{\ww d_t}\sum_{i=1}^m\Big[\|x_{i,t}-\bar x_t\|+\|\bar x_t-x_{\ell,t}\|\Big].
\eeal
\ees
Then it can be observed that the  term $3G\sum_{t=1}^T\f{\aaa_t}{\ww d_t}\sum_{i=1}^m\Big[\|x_{i,t}-\bar x_t\|+\|\bar x_t-x_{\ell,t}\|\Big]$ of the above inequality shares the same bound with \eqref{xlt_yit} up to a scaling constant $3G$. Then, following the same procedures after \eqref{xlt_yit} in Proposition \ref{RT_second_term_est}, we arrive at, for any $S\in[T]$ and $\ell\in\huaV$, with probability at least $1-\delta_5-\delta_6-\delta_7$,
\bea
\beal
\ww \huaQ_T\leq&\ww C_3+\ww C_4\sum_{t=1}^T\aaa_t^2+\Big[\ww C_5\left(\log\f{2T}{\delta_5}\right)^\ttt\\
&+\ww C_6\left(\log\f{2}{\delta_6}\right)^\ttt+\ww C_7\left(\log\f{2}{\delta_7}\right)^\ttt\Big]\sum_{t=1}^T\aaa_t^2,\\
=:&\huaZ_\huaQ\left(\delta_5,\delta_6,\delta_7,(\aaa_t)_{t=1}^T\right),  \label{Qpiao_est}
\eeal  
\eea
and hence,
\bes
\huaQ_S\geq\sum_{t=1}^S\f{\aaa_t}{\ww d_t}\big[f(x_{\ell,t})-f(x^*)\big]-\huaZ_\huaQ\left(\delta_5,\delta_6,\delta_7,(\aaa_t)_{t=1}^T\right)
\ees
in which $\ww C_i=\f{3G}{G_\psi}C_i$, $i=3,4$ and $\ww C_j=\f{3G}{G_\psi}C_j'$, $j=5,6,7$ with $C_3$, $C_4$, $C_5'$, $C_6'$, $C_7'$ defined as above.
\end{proof}
\subsection{Estimates on $\huaR_S$}
\begin{pro} \label{R_S_est}
Under Assumption \ref{as1},	for any $S\in[T]$, it holds that
	\bes
	&&\huaR_S\leq\f{1}{d_1}\sum_{j=1}^mD_{\ff}(x^*\|x_{j,1})-\f{1}{\ww d_{S}}\sum_{j=1}^mD_{\ff}(x^*\|x_{j,S+1}).
	\ees
\end{pro}
 \begin{proof}
 	According to the algorithm structure and double stochasticity of the communication matrix $W_t$, we have
\bes
&&\hspace{-0.5cm}\huaR_S=\sum_{t=1}^S\f{1}{\ww d_t}\sum_{i=1}^m\Big[D_{\ff}(x^*\|\sum_{j=1}^m[W_t]_{ij}x_{j,t})-D_{\ff}(x^*\|x_{i,t+1})\Big]\\
&&\hspace{-0.5cm}\leq\sum_{t=1}^S\f{1}{\ww d_t}\sum_{i=1}^m\Big[\sum_{j=1}^m[W_t]_{ij}D_{\ff}(x^*\|x_{j,t})-D_{\ff}(x^*\|x_{i,t+1})\Big]\\
&&\hspace{-0.5cm}=\sum_{t=1}^S\f{1}{\ww d_t}\Big[\sum_{j=1}^mD_{\ff}(x^*\|x_{j,t})-\sum_{i=1}^mD_{\ff}(x^*\|x_{i,t+1})\Big].
\ees
The above inequality can be further bounded by
\bes
&&\f{1}{\ww d_1}\sum_{j=1}^mD_{\ff}(x^*\|x_{j,1})-\f{1}{\ww d_S}\sum_{i=1}^mD_{\ff}(x^*\|x_{i,S+1})\\
&&+\sum_{t=2}^S\left(\f{1}{\ww d_{t}}-\f{1}{\ww d_{t-1}}\right)\sum_{i=1}^m D_{\ff}(x^*\|x_{i,t}).
\ees
We know that, for  non-decreasing sequence $\{\ww d_t\}$, it holds that $\f{1}{\ww d_{t}}-\f{1}{\ww d_{t-1}}\leq0$, $t\geq2$. Hence, after simplification, we have 
\bes
\huaR_S\leq\f{1}{\ww d_1}\sum_{j=1}^mD_{\ff}(x^*\|x_{j,1})-\f{1}{\ww d_S}\sum_{i=1}^mD_{\ff}(x^*\|x_{i,S+1}),
\ees
which completes the proof after noting that $d_1\leq\ww d_1$.
\end{proof}

\subsection{High probability estimates for $\huaS_S$}
The following inspiring lemma on the martingale difference sequence concentration inequality for sub-Weibull random variables will be utilized for deriving core high-probability estimate of $\huaS_S$.
\begin{lem}\label{Freedman_ine}
(Theorem 11 in Reference \cite{mdb2024})	Let ($\Omega$, $\huaF$, $\{\huaF_t\}$, $\mbb P$) be a filtered probability space. Let $\{\zeta_t\}$ and $\{\huaK_t\}$ be adapted to $\{\huaF_t\}$. Let $T\in \mbb N_+$. For $t\in[T]$, assume $\huaK_{t-1}>0$ almost surely, $\mbb E[\zeta_t|\huaF_{t-1}]=0$, and
	\bes
	\mbb E\left[\exp\left\{\left(\f{|\zeta_t|}{\huaK_{t-1}}\right)^{1/\ttt}\right\}\Big|\huaF_{t-1}\right]\leq2,
	\ees
		where $\ttt\geq1/2$. Assume there exist constants $\{\huaM_t\}$ such that $\huaK_{t-1}\leq\huaM_t$ almost surely for all $t\in[T]$. Let $\delta\in(0,1)$. Define 
		\bes
		a=\left\{
		\begin{array}{ll}
			2, & \hbox{$\ttt=1/2$;} \\
			(4\ttt)^{2\ttt}e^2, & \hbox{$1/2<\ttt\leq1$;}\\
			(2^{2\ttt+1}+2)\Gamma(2\ttt+1)+\f{2^{3\ttt}\Gamma(3\ttt+1)}{3\log(T/\delta)^{\ttt-1}}, & \hbox{$\ttt>1$,}
		\end{array}
		\right.
		\ees
		\bes
		b=\left\{
		\begin{array}{ll}
			(4\ttt)^\ttt e, & \hbox{$1/2<\ttt\leq1$;} \\
			2\log(T/\delta)^{\ttt-1}, & \hbox{$\ttt>1$.}
		\end{array}
		\right.
		\ees
	We have,	for all $w,\beta\geq0$,  \bes
		\mu\left\{
		\begin{array}{ll}
		>0, & \hbox{$\ttt=1/2$;} \\
			\geq b\max_{t\in[T]}\huaM_t, & \hbox{$\ttt>1/2$;}
		\end{array}
		\right.
		\ees
		and $\lambda\in[0,\f{1}{2\mu}]$, there holds
		\bes
&&\mbb P		\left(\bigcup_{k\in[T]}\left\{\sum_{t=1}^k\zeta_t\geq w \ \text{and} \ \sum_{t=1}^k a \huaK_{t-1}^2\leq\mu\sum_{t=1}^k\zeta_t+\beta\right\}\right)\\
&&\leq\exp(-\lambda w+2\lambda^2\beta)+2\delta.
		\ees
\end{lem}

Here, as mentioned in Lemma \ref{Freedman_ine}, when $\ttt=\f{1}{2}$, $\mu$ can take any positive number. Now we are ready to state the following main high-probability estimate in this subsection.
\begin{pro}\label{huaS_est}
Under Assumptions \ref{subweibull_ass}, \ref{as2} and \ref{as1}. Let $0<\aaa_t<1$ be a non-increasing stepsize. For any $S\in[T]$, we have,	for any $\delta_i\in(0,1)$, $i=8,9$ with $\sum_{i=8}^9\delta_i<1$, with probability at least $1-\delta_8-\delta_9$,
	\bes
	&&\hspace{-0.5cm}(\sup_{S\in[T]}\huaS_S)_+\leq\ww C_8\left(\log\f{2T}{\delta_8}\right)^{\max\{0,\ttt-1\}}\log\f{1}{\delta_9}+\ww C_9\sum_{t=1}^T\aaa_t^2 ,\\
	&&\hspace{-0.5cm}=:\huaZ_\huaS(\delta_8,\delta_9,(\aaa_t)_{t=1}^T),\ \ttt\geq1/2.
	\ees
	with $\ww C_8$ and $\ww C_9$ defined as  $\ww C_8=2\ (\ttt=1/2);2(4\ttt)^\ttt e\sqrt{\f{2}{\sss_\ff}}m^{\ttt+1}\kkk \ (\ttt\in(1/2,1]);4\sqrt{\f{2}{\sss_\ff}}m^{\ttt+1}\kkk \ (\ttt\in(1,\infty))$ and $\ww C_9=\f{2\ww a}{\sss_\ff}m^{2\ttt+2}\kkk^2 \ (\ttt=1/2); 	\f{\ww a\sqrt{\f{2}{\sss_\ff}}m^{\ttt+1}\kkk }{(4\ttt)^\ttt e} \ (\ttt\in(1/2,1]); \f{\ww a}{2}\sqrt{\f{2}{\sss_\ff}}m^{\ttt+1}\kkk \ (\ttt\in(1,\infty))$. Here, $\ww a$ is a constant after replacing $\log(T/\delta)^{\ttt-1}$ in $a$ with $1$.
	\end{pro}
\begin{proof}
The $\sss_\ff$-strong convexity of  $\Phi$ implies that, for any $j\in\huaV$, there holds $\ff(x_{j,t})-\ff(x^*)-\nn\nabla \ff(x^*),x_{j,t}-x^*\mm\geq\f{\sss_\ff}{2}\|x_{j,t}-x^*\|^2.$
This directly implies 
$\|x_{j,t}-x^*\|\leq\sqrt{\f{2}{\sss_\ff}}\ww d_t, j\in\huaV$.
Then it follows that $\|y_{i,t}-x^*\|=\|\sum_{j=1}^m[W_t]_{ij}x_{j,t}-x^*\|\leq\sum_{j=1}^m[W_t]_{ij}\|x_{j,t}-x^*\|\leq\sqrt{\f{2}{\sss_\ff}}\ww d_t$.
Then we arrive at $\left|\f{\aaa_t}{\ww d_t}\sum_{i=1}^m\nn  \xi_{i,t},y_{i,t}-x^*\mm\right|\leq\aaa_t\sum_{i=1}^m\left\|\f{y_{i,t}-x^*}{\ww d_t}\right\| \|\xi_{i,t}\|_*\leq\sqrt{\f{2}{\sss_\ff}}\aaa_t\sum_{i=1}^m\|\xi_{i,t}\|_*.$
Taking expectation on both sides of this inequality, we have $\mbb E[|\f{\aaa_t}{\ww d_t}\sum_{i=1}^m\nn  \xi_{i,t},y_{i,t}-x^*\mm|] \leq\sqrt{\f{2}{\sigma_\ff}}\aaa_t\mbb E[\sum_{i=1}^m\|\xi_{i,t}\|_*]$. Lemma \ref{summation_SubW} indicates that $\sum_{i=1}^m\|\xi_{i,t}\|_*\sim SubW(\ttt,K_\ttt m\kkk)$, where $K_\ttt=m^\ttt$ for $\ttt>1$ and $K_\ttt=1$ for $0<\ttt\leq1$. Then applying Lemma \ref{moment_subweibull} to $X=\sum_{i=1}^m\|\xi_{i,t}\|_*$ and $p=1$, we have $\mbb E\left[\sum_{i=1}^m\|\xi_{i,t}\|_*\right]\leq2\Gamma(\ttt+1)K_\ttt m\kkk<\infty$. Hence $\mbb E[|\f{\aaa_t}{\ww d_t}\sum_{i=1}^m\nn  \xi_{i,t},y_{i,t}-x^*\mm|]<\infty$. On the other hand, $\mbb E[\f{\aaa_t}{\ww d_t}\sum_{i=1}^m\nn  \xi_{i,t},y_{i,t}-x^*\mm|\huaF_{t-1}]=\aaa_t\sum_{i=1}^m\nn  \mbb E[\xi_{i,t}|\huaF_{t-1}],\f{y_{i,t}-x^*}{\ww d_t}\mm=0$ ($\ww d_t$ is measurable w.r.t. $\huaF_{t-1}$). Hence $\f{\aaa_t}{\ww d_t}\sum_{i=1}^m\nn  \xi_{i,t},y_{i,t}-x^*\mm$ is a martingale difference sequence. Denote $\huaK_{t-1}=\huaM_t=\sqrt{\f{2}{\sss_\ff}}m^{\ttt+1}\kkk\aaa_t$, then, for any $\ttt\geq\f{1}{2}$, there always holds that
\bes
&&\mbb E\left[\exp\left\{\left(\f{\big|\f{\aaa_t}{\ww d_t}\sum_{i=1}^m\nn  \xi_{i,t},y_{i,t}-x^*\mm\big|}{\huaK_{t-1}}\right)^{1/\ttt}\right\}\Bigg|\huaF_{t-1}\right]\\
&&\leq\mbb E\left[\exp\left\{\left(\f{\sum_{i=1}^m\|\xi_{i,t}\|_*}{m^{\ttt+1}\kkk}\right)^{1/\ttt}\right\}\Bigg|\huaF_{t-1}\right]\leq2.
\ees
That is to say, $\f{\aaa_t}{\ww d_t}\sum_{i=1}^m\nn  \xi_{i,t},y_{i,t}-x^*\mm$ is a sub-Weibull martingale difference sequence of which the scale parameters are controlled by $\huaK_{t-1}$. Set $\delta_8, \delta_9\in(0,1)$,  let the $\delta$ in Lemma \ref{Freedman_ine} be $\delta_8$ and $\beta=0$, set $\lambda=\f{1}{2\mu}$, $w=2\mu\log\f{1}{\delta_9}$ and $\mu=1\ (\ttt=1/2); \ (4\ttt)^\ttt e\sqrt{\f{2}{\sss_\ff}}m^{\ttt+1}\kkk\ (\ttt\in(1/2,1]); \ 2\left(\log\f{T}{\delta_8}\right)^{\ttt-1}\sqrt{\f{2}{\sss_\ff}}m^{\ttt+1}\kkk\ (\ttt\in(1,\infty))$. 
 Note that, for any $S\in[T]$,
 \bes
 &&\mbb P\left(\sum_{t=1}^S\f{\aaa_t}{\ww d_t}\sum_{i=1}^m\nn  \xi_{i,t},y_{i,t}-x^*\mm\geq w+\f{a}{\mu}\sum_{t=1}^T\huaK_{t-1}^2\right)\\
 &&\leq\mbb P\left(\sum_{t=1}^S\f{\aaa_t}{\ww d_t}\sum_{i=1}^m\nn  \xi_{i,t},y_{i,t}-x^*\mm\geq w+\f{a}{\mu}\sum_{t=1}^S\huaK_{t-1}^2\right)\\
 &&\leq\mbb P\Bigg(\sum_{t=1}^S\f{\aaa_t}{\ww d_t}\sum_{i=1}^m\nn  \xi_{i,t},y_{i,t}-x^*\mm\geq w \ \text{and}\\
 &&\quad\quad\sum_{t=1}^S\huaK_{t-1}^2\leq\mu\sum_{t=1}^S\f{\aaa_t}{\ww d_t}\sum_{i=1}^m\nn  \xi_{i,t},y_{i,t}-x^*\mm\Bigg).
 \ees
Then by utilizing Lemma \ref{Freedman_ine}, we have
\bes
&&\hspace{-0.5cm}\mbb P\left(\sup_{S\in[T]}\left\{\sum_{t=1}^S\f{\aaa_t}{\ww d_t}\sum_{i=1}^m\nn  \xi_{i,t},y_{i,t}-x^*\mm\right\}\geq w+\f{a}{\mu}\sum_{t=1}^T\huaK_{t-1}^2\right)\\
&&\hspace{-0.5cm}\leq\mbb P\Bigg(\bigcup_{S\in[T]}\Bigg\{\sum_{t=1}^S\f{\aaa_t}{\ww d_t}\sum_{i=1}^m\nn  \xi_{i,t},y_{i,t}-x^*\mm\geq w \ \text{and}\\
&&\hspace{-0.5cm}\quad\quad\sum_{t=1}^S\huaK_{t-1}^2\leq\mu\sum_{t=1}^S\f{\aaa_t}{\ww d_t}\sum_{i=1}^m\nn  \xi_{i,t},y_{i,t}-x^*\mm\Bigg\}\Bigg)\\
&&\hspace{-0.5cm}\leq\exp(-\lambda w)+2\delta_8.
\ees
Noting that, for any $r>0$, it holds that $\mbb P(\sup_{S\in[T]}\{\sum_{t=1}^S\f{\aaa_t}{\ww d_t}\sum_{i=1}^m\nn  \xi_{i,t},y_{i,t}-x^*\mm\}\geq r)=\mbb P((\sup_{S\in[T]}\{\sum_{t=1}^S\f{\aaa_t}{\ww d_t}\sum_{i=1}^m\nn  \xi_{i,t},y_{i,t}-x^*\mm\})_+\geq r)$.
We have, with probability at least $1-2\delta_8-\delta_9$,
\bes
	&&\hspace{-1cm}\Big(\sup_{S\in[T]}\huaS_S\Big)_+=\left(\sup_{S\in[T]}\left\{\sum_{t=1}^S\f{\aaa_t}{\ww d_t}\sum_{i=1}^m\nn  \xi_{i,t},y_{i,t}-x^*\mm\right\}\right)_+\\\
	&&\hspace{-1cm}	\leq\ww C_8\left(\log\f{2T}{\delta_8}\right)^{\max\{0,\ttt-1\}}\log\f{1}{\delta_9}+\ww C_9\sum_{t=1}^T\aaa_t^2 ,\quad \ttt\geq\f{1}{2},
\ees
with $\ww C_8$ and $\ww C_9$ defined as above.
\end{proof}

\subsection{High probability estimates for $\huaT_S$}
According to the definition of $\ww d_t$ in \eqref{piao_d_def},  for any $\rho>0$,
\bes
\huaT_S\leq\rho\vee\left(\f{1}{\rho}\ww \huaT_T\right) \ \text{with} \ \ww \huaT_T=\sum_{t=1}^T\f{\aaa_t^2}{2\sigma_\Phi }\sum_{i=1}^m\left\|\wh g_{i,t}^{SW}\right\|_*^2.
\ees
\begin{pro}\label{huaT_est_pro}
Under Assumptions   \ref{subweibull_ass}, \ref{gradient_bdd} and \ref{as1}. For any $S\in[T]$, we have, for any $\delta_{10}\in(0,1)$,	with probability at least $1-\delta_{10}$,
	\bes
	\beal
	\ww \huaT_T\leq&\ww C_{10}\sum_{t=1}^T\aaa_t^2+\ww C_{11}\left(\log\f{2}{\delta_{10}}\right)^{2\ttt}\sum_{t=1}^T\aaa_t^2\\
	=:&\huaZ_\huaT(\delta_{10},(\aaa_t)_{t=1}^T) \eeal
	\ees
	with  $\ww C_{10}=\f{1}{\sss_\ff}mG^2$ and $\ww C_{11}=\f{1}{\sss_\ff}v_{2\ttt}m^{2\ttt+1}\kkk^2$.
\end{pro}
\begin{proof}
The representation $\wh g_{i,t}^{SW}=g_{i,t}-\xi_{i,t}$ together with   Assumption \ref{gradient_bdd}   yield that, for any $S\in[T]$,
\bea
\ww\huaT_T\leq\f{mG^2}{\sss_\ff}\sum_{t=1}^T\aaa_t^2+\f{1}{\sss_\ff}\sum_{t=1}^T\aaa_t^2\sum_{i=1}^m\wh \xi_{i,t}. \label{huaT_eq1}
\eea
with $\wh \xi_{i,t}$ defined in \eqref{xi_def2}. Since $\|\xi_{i,t}\|_*\sim SubW(\ttt,\kkk)$, we know  $\wh \xi_{i,t}\sim SubW(2\ttt,\kkk^2)$, $\ttt\geq1/2$. Then it follows that
$\sum_{i=1}^m\wh \xi_{i,t}\sim SubW(2\ttt,m^{2\ttt+1} \kkk^2)$. Then we have $\aaa_t^2\sum_{i=1}^m\wh \xi_{i,t}\sim SubW(2\ttt,m^{2\ttt+1} \kkk^2\aaa_t^2)$. Applying Lemma \ref{subweibull_concentration}, we obtain that $\mbb P\left(\left|\sum_{t=1}^T\aaa_t^2\sum_{i=1}^m\wh \xi_{i,t}\right|\geq r\right)\leq2\exp\left\{-\left(\f{r}{v_{2\ttt}m^{2\ttt+1}\kkk^2\sum_{t=1}^T\aaa_t^2}\right)^{1/2\ttt}\right\}$.
Hence we have, with probability at least $1-\delta_{10}$,
\bes
\sum_{t=1}^T\aaa_t^2\sum_{i=1}^m\wh\xi_{i,t}\leq  v_{2\ttt}m^{2\ttt+1}\kkk^2\left(\log\f{2}{\delta_{10}}\right)^{2\ttt}\sum_{t=1}^T\aaa_t^2. 
\ees
Combining this inequality with \eqref{huaT_eq1} and setting $\ww C_{10}=\f{1}{\sss_\ff}mG^2$ and $\ww C_{11}=\f{1}{\sss_\ff}v_{2\ttt}m^{2\ttt+1}\kkk^2$, we obtain the desired result.
\end{proof}
\section{Proofs of main theorems}
\begin{proof}[Proof of Theorem \ref{main_general_bdd}]
By combining the equation \eqref{orignal_est} and Propositions \ref{huaP_est_pro}, \ref{huaQ_est_pro}, \ref{R_S_est}, \ref{huaS_est}, \ref{huaT_est_pro}, setting $\delta=\delta_1=\cdots=\delta_{10}$,  $\bm{\huaZ}_\huaP\left(\delta,(\aaa_t)_{t=1}^T\right)=\huaZ_\huaP\left(\delta_1,\delta_2,\delta_3,\delta_4,(\aaa_t)_{t=1}^T\right)$, $\bm{\huaZ}_\huaQ\left(\delta,(\aaa_t)_{t=1}^T\right)=\huaZ_\huaQ\left(\delta_5,\delta_6,\delta_7,(\aaa_t)_{t=1}^T\right)$, $\bm{\huaZ}_\huaS\left(\delta,(\aaa_t)_{t=1}^T\right)=\huaZ_\huaS(\delta_8,\delta_9,(\aaa_t)_{t=1}^T)$, $\bm{\huaZ}_\huaT\left(\delta,(\aaa_t)_{t=1}^T\right)=\huaZ_\huaT(\delta_{10},(\aaa_t)_{t=1}^T)$ and $B_1=C_3$, $B_2=C_1+C_4$, $B_3=C_2+C_6'+C_7'$, $B_4=C_5'$, $B_5=\ww C_3$, $B_6=\ww C_4$, $B_7=\ww C_6+\ww C_7$, $B_8=\ww C_5$, $B_9=\ww C_8$, $B_{10}=\ww C_9$, $B_{11}=\ww C_{10}$, $B_{12}=\ww C_{11}$, the desired result then follows directly.
\end{proof}

\begin{proof}[Proof of Theorem \ref{main_eqi_weight}]
	Before deriving the core high-probability bound and rate for $F(\ww x_{\ell}^T)-F(x^*)$, from the byproducts of the analysis process above, after combining the estimates \eqref{bdd_supportP}, \eqref{bdd_supportQ}, Proposition \ref{R_S_est} with \eqref{orignal_est} we obtain that, for any $\ell\in\huaV$, $S\in[T]$ and $\rho>0$,
\bea
\nono&&\f{1}{\ww d_{S}}\sum_{j=1}^mD_{\ff}(x^*\|x_{j,S+1})+\sum_{t=1}^S\f{\aaa_t}{\ww d_t}\Big[F(x_{\ell,t})-F(x^*)\Big]\\
\nono&&\leq\f{1}{d_1}\sum_{j=1}^mD_{\ff}(x^*\|x_{j,1})+\ww \huaP_T+\ww \huaQ_T\\
&&\quad +\Big(\sup_{S\in[T]}\huaS_S\Big)_++\rho\vee\left(\f{1}{\rho}\ww \huaT_T\right)=:\huaB_T. \label{main_equation}
\eea
Here, $\ww \huaP_T$ and $\ww \huaQ_T$ are defined in \eqref{RT_start} and \eqref{wwhuaQ_def}. Since $F(x_{\ell,t})-F(x^*)\geq0$, then we know that $\f{d_{S+1}^2}{\ww d_S}\leq \f{1}{\ww d_{S}}\sum_{j=1}^mD_{\ff}(x^*\|x_{j,S+1})\leq\huaB_T$,
namely $d_{S+1}^2\leq\huaB_T\ww d_S$. It is easy to see $\ww d_1=\max\{d_1,\rho\}\leq\huaB_T$. If for some $S\in[T]$, $\ww d_S\leq\huaB_T$, then $\ww d_{S+1}=\max\{\ww d_S,d_{S+1}\}\leq\huaB_T$. Hence, an induction process indicates that $\ww d_S\leq\huaB_T$, $S\in[T]$. Then we have $\f{\aaa_T}{\ww d_T}\sum_{t=1}^T\Big[F( x_{\ell,t})-F(x^*)\Big]\leq\huaB_T$. As a result of the convexity of $F$, we arrive at, for any $\ell\in\huaV$, $F(\ww x_{\ell}^T)-F(x^*)\leq\f{\huaB_T^2}{T \aaa_T}$.
According to previous analysis, for any $0<\delta<\f{1}{10}$, after taking $\rho=\bm{\huaZ}_\huaT(\delta,(\aaa_t)_{t=1}^T)^{1/2}$, we have, with probability at least $1-10\delta$,
\bes
&&\hspace{-0.5cm}\huaB_T\leq \f{1}{d_1}\sum_{j=1}^mD_{\ff}(x^*\|x_{j,1})+\bm{\huaZ}_\huaP\left(\delta,(\aaa_t)_{t=1}^T\right)\\
&&\hspace{-0.5cm}+\bm{\huaZ}_\huaQ\left(\delta,(\aaa_t)_{t=1}^T\right)+\bm{\huaZ}_\huaS\left(\delta,(\aaa_t)_{t=1}^T\right)+\bm{\huaZ}_\huaT\left(\delta,(\aaa_t)_{t=1}^T\right)^{1/2}
\ees
with $\bm{\huaZ}_\huaP\left(\delta,(\aaa_t)_{t=1}^T\right)$, $\bm{\huaZ}_\huaQ\left(\delta,(\aaa_t)_{t=1}^T\right)$,  $\bm{\huaZ}_\huaS\left(\delta,(\aaa_t)_{t=1}^T\right)$, $\bm{\huaZ}_\huaT\left(\delta,(\aaa_t)_{t=1}^T\right)$ defined in Theorem \ref{main_general_bdd}.
Since, $\aaa_t=\f{1}{\sqrt{t+1}}$, then we know $\sum_{t=1}^T\aaa_t^2\leq\log(T+1)$. Accordingly, $\bm{\huaZ}_\huaP\left(\delta,(\aaa_t)_{t=1}^T\right)$ can be further bounded as $\bm{\huaZ}_\huaP\left(\delta,(\aaa_t)_{t=1}^T\right)\leq B_1+\Big[B_2+B_3\left(\log\f{2}{\delta}\right)^\ttt+B_4\left(\log\f{2T}{\delta}\right)^\ttt\Big]\log(T+1)$,
$\bm{\huaZ}_\huaQ\left(\delta,(\aaa_t)_{t=1}^T\right)$ can be further bounded as $\bm{\huaZ}_\huaQ\left(\delta,(\aaa_t)_{t=1}^T\right)\leq B_5+\Big[B_6+B_7\left(\log\f{2}{\delta}\right)^\ttt+B_8\left(\log\f{2T}{\delta}\right)^\ttt\Big]\log(T+1)$,
$\bm{\huaZ}_\huaS\left(\delta,(\aaa_t)_{t=1}^T\right)$ can be further bounded as $\bm{\huaZ}_\huaS\left(\delta,(\aaa_t)_{t=1}^T\right)\leq B_9\left(\log\f{2T}{\delta}\right)^{\max\{0,\ttt-1\}}\log\f{1}{\delta}+B_{10}\log(T+1), \ttt\geq\f{1}{2}$,
$\bm{\huaZ}_\huaT(\delta,(\aaa_t)_{t=1}^T)^{1/2}$ can be further bounded as $\bm{\huaZ}_\huaT(\delta,(\aaa_t)_{t=1}^T)^{1/2}\leq\Big(\Big[B_{11}+B_{12}\left(\log\f{2}{\delta}\right)^{2\ttt}\Big]\log(T+1)\Big)^{1/2}$.
Then  it follows that, with probability at least $1-\delta$, $\huaB_T\leq\huaO(\log T+\left(\log\f{T}{\delta}\right)^\ttt\log T+\left(\log\f{1}{\delta}\right)^\ttt\log T+\sqrt{\log T}+\left(\log\f{1}{\delta}\right)^\ttt\sqrt{\log T}+ \left(\log\f{1}{\delta}\right)\left(\log\f{T}{\delta}\right)^{\max\{0,\ttt-1\}})\leq\huaO(\left(\log\f{T}{\delta}\right)^\ttt\log T+\left(\log\f{1}{\delta}\right)\left(\log\f{T}{\delta}\right)^{\max\{0,\ttt-1\}})$.
Hence we have, with probability at least $1-\delta$, $\huaB_T^2\leq\huaO\big(\left(\log\f{T}{\delta}\right)^{2\ttt}(\log T)^2+\left(\log\f{1}{\delta}\right)^2\left(\log\f{T}{\delta}\right)^{2\max\{0,\ttt-1\}}\big)$.
Accordingly, we arrive at, for any $\ell\in\huaV$, after selecting the stepsize $\aaa_t=\huaO(1/\sqrt{t})$, then with probability $1-\delta$,
\bes
&&F(\ww x_{\ell}^T)-F(x^*)\leq\huaO\Bigg(\left(\log\f{T}{\delta}\right)^{2\ttt}\f{(\log T)^2}{\sqrt{T}}\\
&&\quad +\left(\log\f{1}{\delta}\right)^2\left(\log\f{T}{\delta}\right)^{2\max\{0,\ttt-1\}}\f{1}{\sqrt{T}}\Bigg), \ \ttt\geq\f{1}{2}.
\ees
The proof is complete.
\end{proof}

\begin{proof}[Proof of Theorem \ref{main_thm_nonequi_weight}]
	As a byproduct of the equation \eqref{main_equation}, we are able to witness that, for any $S\in[T]$,
	$\sum_{t=1}^S\f{\aaa_t}{\ww d_T}\Big[F(x_{\ell,t})-F(x^*)\Big]\leq\huaB_T$.
		When $S=T$, it follows that $	\f{1}{\sum_{t=1}^T\aaa_t}\sum_{t=1}^T\aaa_t\Big[F(x_{\ell,t})-F(x^*)\Big]\leq\f{\huaB_T\ww d_T}{\sum_{t=1}^T\aaa_t}$.
		Using the fact that $\ww d_T\leq \huaB_T$ and the convexity of $F$, we have, for any $\ell\in\huaV$, the ergodic sequence $\wh x_\ell^T=\f{\sum_{t=1}^T\aaa_tx_{\ell,t}}{\sum_{t=1}^T\aaa_t}$ satisfies $F(\wh x_\ell^T)-F(x^*)\leq\f{\huaB_T^2}{\sum_{t=1}^T\aaa_t}$,  $\ell\in\huaV$.
		Due to the fact that $\sum_{t=1}^T\aaa_t\geq\huaO(\sqrt{T})$, combining this with the high-probability bound of $\huaB_T$, we have, with probability at least $1-\delta$,
		\bes
		&&F(\wh x_{\ell}^T)-F(x^*)\leq\huaO\Bigg(\left(\log\f{T}{\delta}\right)^{2\ttt}\f{(\log T)^2}{\sqrt{T}}\\
		&&\quad +\left(\log\f{1}{\delta}\right)^2\left(\log\f{T}{\delta}\right)^{2\max\{0,\ttt-1\}}\f{1}{\sqrt{T}}\Bigg), \ \ttt\geq\f{1}{2}.
		\ees
		The proof is complete.
	\end{proof}
\begin{proof}[Proof of Theorem \ref{main_thm_known_time}]
To prove the result, we only need to note that when $\aaa_t=1/\sqrt{T}$, $t=1,2,...,T$, the corresponding $\sum_{t=1}^T\aaa_t^2=\huaO(1)$ in $\bm{\huaZ}_\huaP\left(\delta,(\aaa_t)_{t=1}^T\right)$, $\bm{\huaZ}_\huaQ\left(\delta,(\aaa_t)_{t=1}^T\right)$,  $\bm{\huaZ}_\huaS\left(\delta,(\aaa_t)_{t=1}^T\right))$, and $\bm{\huaZ}_\huaT\left(\delta,(\aaa_t)_{t=1}^T\right)$.   Following the same procedures as in the proof of Theorem \ref{main_eqi_weight}, we obtain the desired result.
\end{proof}
\begin{proof}[Proof of Corollary \ref{cor1}]
	We only need to verify the separate convexity of the Bregman divergence $D_\Phi(x\|y)=\f{1}{2}\|x-y\|_2^2$ in terms of variable $y$, this is easily verified after noticing that, for any $y_j\in\huaX$, $j\in\huaV$ and any $\sum_{j=1}^ma_j=1$, $a_j\geq0$, we have $\f{1}{2}\|x-\sum_{j=1}^ma_jy_j\|_2^2\leq\f{1}{2}\sum_{j=1}^ma_j\|x-y_j\|_2^2$ as a result of the convexity of $\f{1}{2}\|\cdot\|_2^2$. Hence Assumption \ref{as1} holds. For the other proofs, they follow directly from other proofs of Theorems \ref{main_general_bdd}, \ref{main_eqi_weight}, and \ref{main_thm_known_time}.	
\end{proof}

\section{Numerical experiments}

\begin{figure*}[t] 
	\centering
	\begin{minipage}{0.24\textwidth}  
		\centering
		\includegraphics[width=\linewidth]{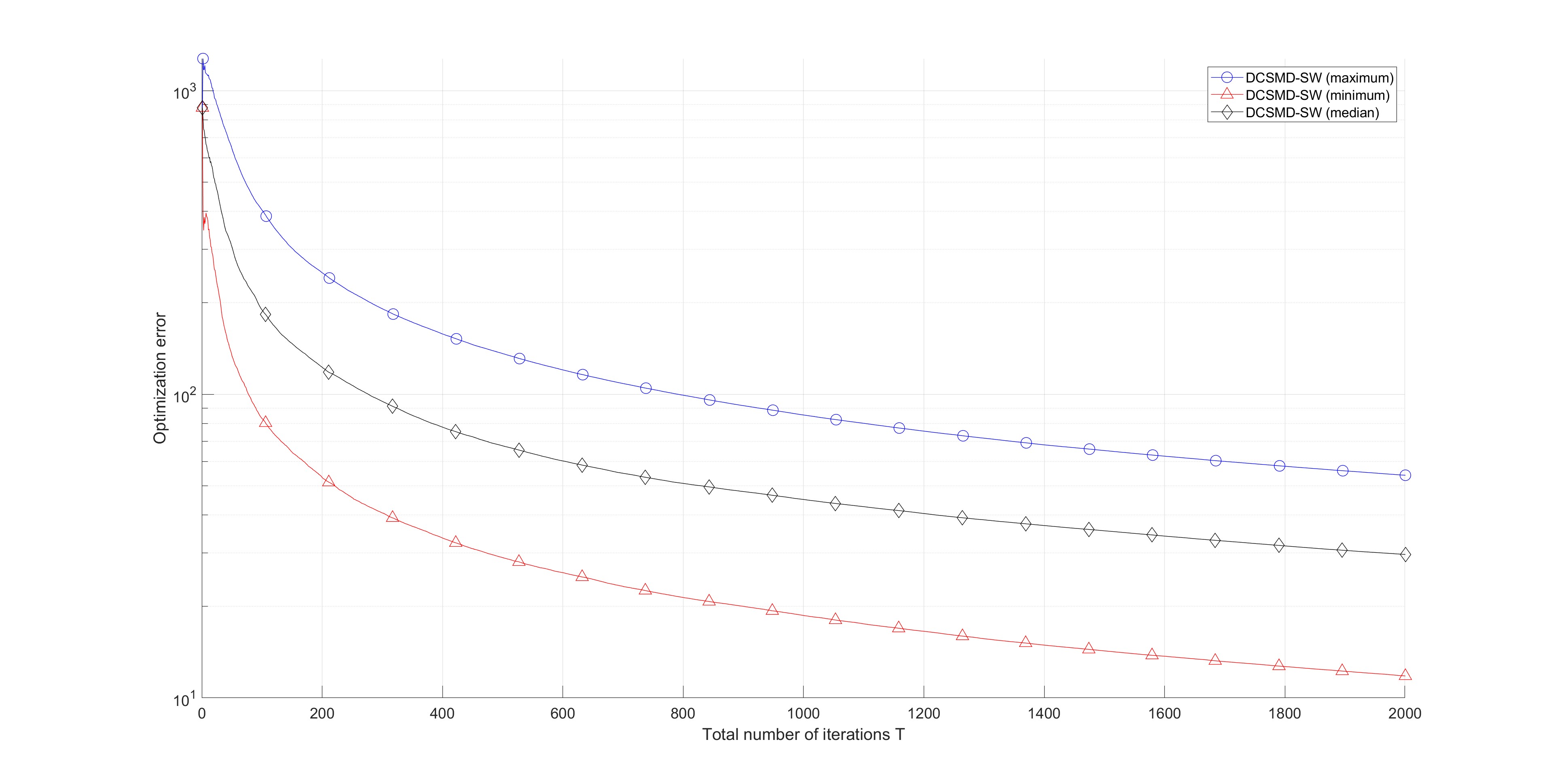}
		\caption{\tiny{Maximum, minimum, median of optimization errors versus total number
				of iterations $T$ for DCSMD-SW.}}
		\label{figure1}
	\end{minipage}
	\hfill 
	\begin{minipage}{0.24\textwidth}
		\centering
		\includegraphics[width=\linewidth]{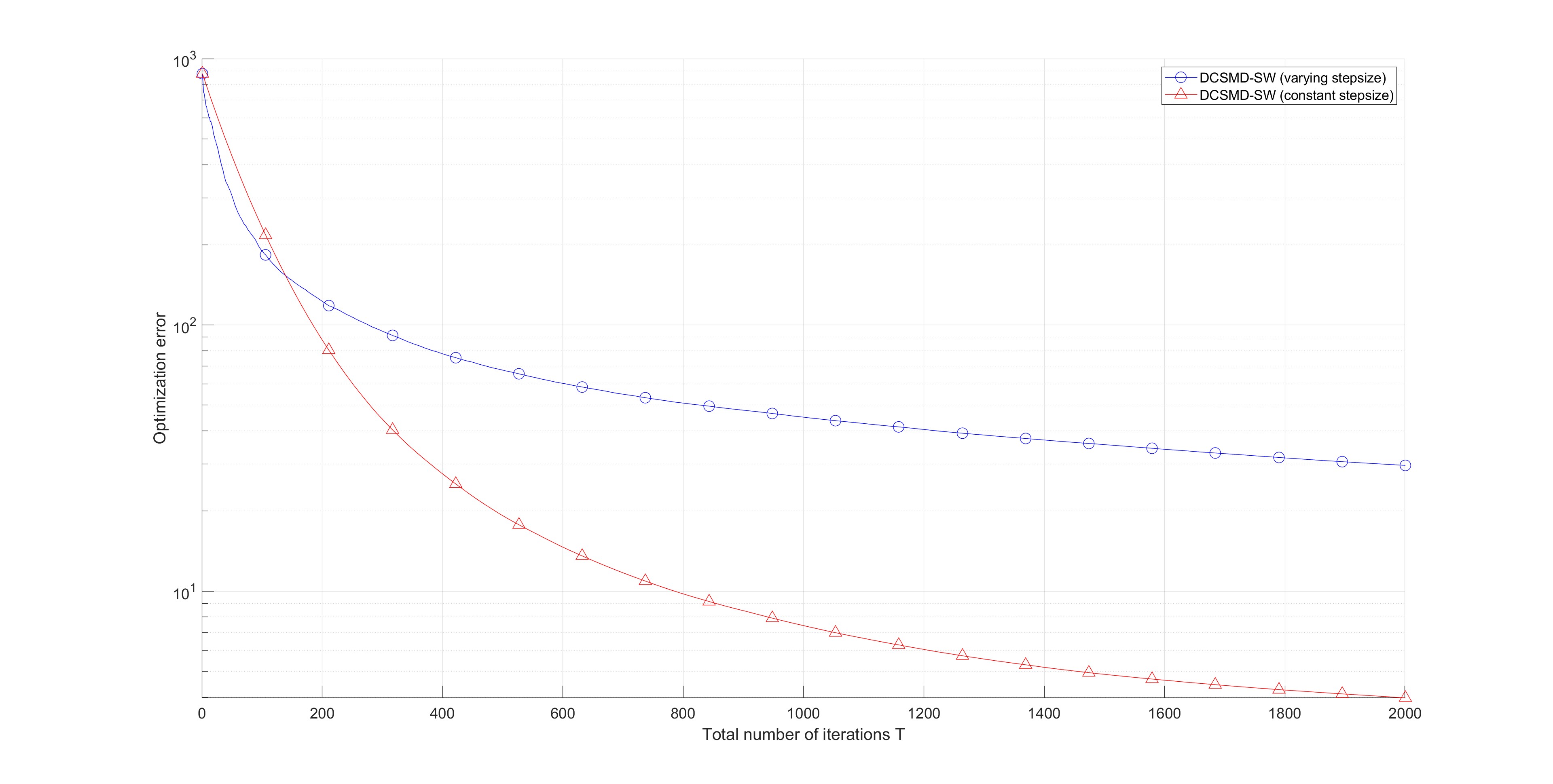}
		\caption{\tiny{Median of optimization errors across $m$ agents versus total number of iterations $T$ for DCSMD-SW with different stepsizes.}}
		\label{figure5}
	\end{minipage}
	\hfill
	\begin{minipage}{0.24\textwidth}
		\centering
		\includegraphics[width=\linewidth]{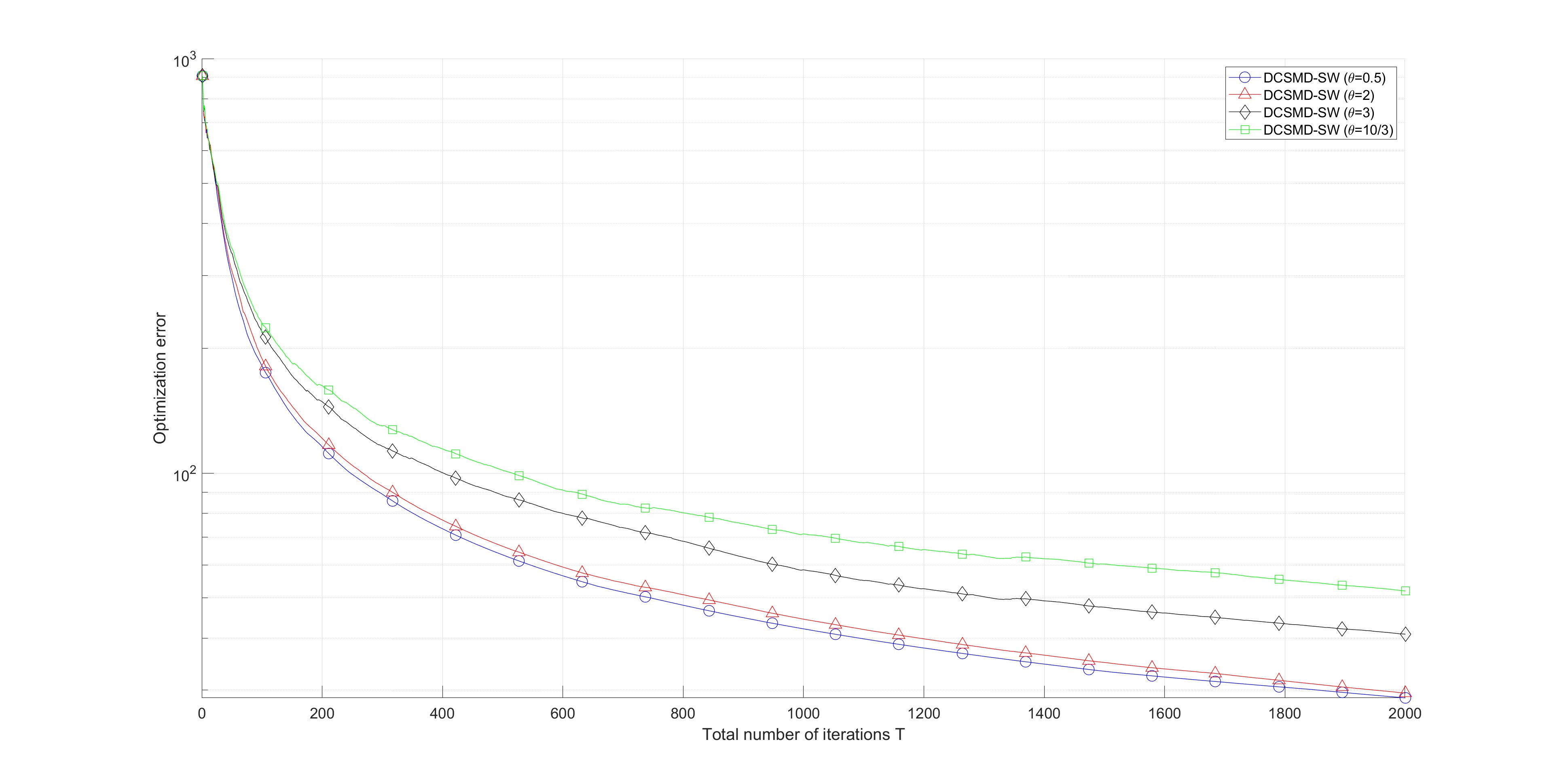}
		\caption{\tiny{Median of optimization errors across $m$ agents versus total number of iterations $T$ for DCSMD-SW with different $\theta$.}}
		\label{regression_theta}
	\end{minipage}
	\hfill
	\begin{minipage}{0.24\textwidth}
		\centering
		\includegraphics[width=\linewidth]{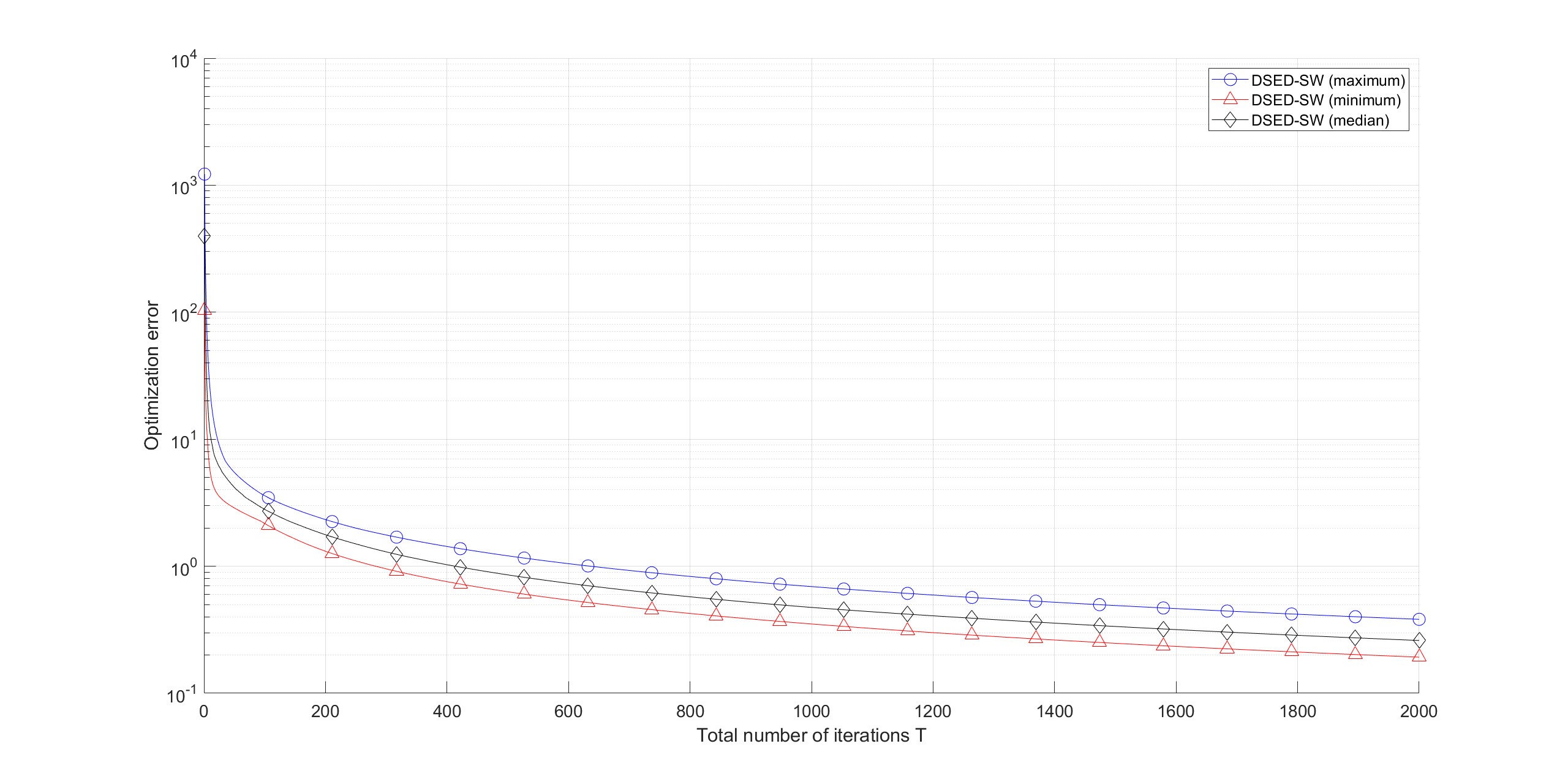}
		\caption{\tiny{Maximum, minimum, and median of optimization errors versus total number
				of iterations $T$ for DSED-SW.}}
		\label{figure7}
	\end{minipage}
\end{figure*}

We perform numerical experiments to demonstrate various characteristics of  DCSMD-SW, with an emphasis on a distributed composite regularized linear regression problem:
\bes
	\min_{x\in\huaX} \sum_{i=1}^m \Big(\sum_{j=n_{i-1}+1}^{n_i} \frac{1}{2}\left(\langle {a}_{j},{x}\rangle-b_{j}\right)^2+\lambda\| x\|_1\Big),
\ees
where $\{a_j \in \mathbb R^n, b_j \in \mathbb R\}_{j=n_{i-1}+1}^{n_i}$ are data known only to agent $i$. Unless stated otherwise, we will adopt the following configurations in our simulation examples. The constraint set is defined as $\huaX=\{x\in\mathbb R^{n}:u_j\leq [x]_j\leq v_j, j=1,2,...,n\}$. The parameters and distance functions in DCSMD-SW are chosen as follows: we consider the number of agents $m=60$, the number of data in each node $n_1= n_2-n_1 = \dots = n_m- n_{m-1}= 10$, $n_0=0$, the problem dimension $n=20$, the hyperparameter $\lambda = 0.1$. The distance generating function is selected as $\Phi(x)=\f{1}{2}\|x\|_2^2$ with $x\in \huaX$. The parameters for the constraints are $u_j=-1$ and $v_j=1$ for $j=1,2,..,n$. In each trial, for $i=1,2,...,m$, each element of the input vector $a_i$ is generated from a uniform distribution over the interval $[-1,1]$, and the response is generated by
$b_i = \langle a_i, x_* \rangle + \epsilon_i$,
where $[x_*]_i=1$ for $1 \leq i \leq \lfloor n/2 \rfloor$ and $0$ otherwise, and the noise $\epsilon_i$ is drawn i.i.d. from a normal distribution $\mathcal N(0,1)$. The stochastic noise $\xi_{i,t}$ in the gradient is generated from a Gaussian distribution $\mathcal{N}(\mathbf{0},I_n)$, which ensures that $\|\xi_{i,t}\|_*$ is a sub-Weibull random variable. The initial values $\{x_{i,1}\}_{i=1}^m$ are sampled from a uniform distribution on $[0,1]^n$. The stepsizes $\{\aaa_k\}$ are selected as $\aaa_k=\f{1}{\sqrt{k+1}}$, $k\geq1$. In our experiments, we evaluate the optimization error  $F(\ww{x}_{\ell}^T)-F(x^{*})$ ($\ell \in \huaV$), where $\ww x_\ell^T=\f{1}{T}\sum_{t=1}^T x_{\ell,t}$ is the local ergodic sequence with equi-weight. All the experiment results are based on the average of 10 runs. We demonstrate the convergence behavior of DCSMD-SW by graphing the maximum, minimum, and median of the optimization errors $\{F(\ww{x}_{\ell}^T)-F(x^{*})\}_{\ell=1}^{m}$ across $m$ agents, plotted against the total number of iterations $T$. The outcomes are presented in Fig. \ref{figure1}, which shows that all agents in DCSMD-SW converge at a satisfactory rate. The similar trends and good convergence performance of the three curves are supported by the fact that the maximum, minimum, and median of the optimization errors $\{F(\ww{x}_{\ell}^T)-F(x^{*})\}_{\ell=1}^{m}$ share the same high-probability convergence rate with $F(\ww{x}_{\ell}^T)-F(x^{*})$ for any $\ell\in\huaV$. To explain this, denote $Z_\ell(T)=F(\ww{x}_{\ell}^T)-F(x^{*})$, $\ell\in\huaV$ and the high-probability convergence bound in Theorem \ref{main_eqi_weight} by $B(T,\delta)$. Then we know, for any $\delta\in(0,1)$, $\mbb P( Z_\ell(T)>B(T,\delta))<\delta$. It follows that $\mbb P(\min_{\ell\in\huaV}  Z_\ell(T)>B(T,\delta))\leq\mbb P( Z_\ell(T)>B(T,\delta))<\delta$. Now we show $\max_{\ell\in\huaV} Z_\ell(T)$ shares the same high-probability rate with $Z_\ell(T)$, $\ell\in\huaV$. Basic probability theory indicates that $\mbb P(\max_{\ell\in\huaV} Z_\ell(T)>B(T,\delta/m))=\mbb P(\bigcup_{\ell\in\huaV}\{Z_\ell(T)>B(T,\delta/m)\})\leq\sum_{\ell\in\huaV}\mbb P( Z_\ell(T)>B(T,\delta/m))<m\cdot\f{\delta}{m}=\delta$. Hence we know $\max_{\ell\in\huaV} Z_\ell(T)$ shares the same high-probability rate with $Z_\ell(T)$ only up to constants. Finally, because
	$\mbb P(\text{med}_{\ell\in\huaV} Z_\ell(T)>B(T,\delta))\leq\mbb P(\max_{\ell\in\huaV} Z_\ell(T)>B(T,\delta))$,  $\text{med}_{\ell\in\huaV} Z_\ell(T)$ shares the same high-probability convergence rate with $\max_{\ell\in\huaV} Z_\ell(T)$.

Next, we explore the selection of a constant stepsize in relation to the time-horizon $T$, specifically $\alpha_k=\huaO(1/\sqrt{T})$. We illustrate the convergence of DCSMD-SW using this constant stepsize and contrast it with a diminishing stepsize $\alpha_k= 1/\sqrt{k+1}$ in Fig. \ref{figure5}. The results indicate that a well-chosen constant stepsize relative to the time-horizon can lead to a slightly faster convergence rate compared to the varying stepsize approach. This aligns with our theoretical conclusions, which suggest that the $(\log T)^2$ order can be eliminated when employing a constant stepsize strategy.

We also investigate the effect of various sub-Weibull stochastic noises $\xi_{i,t}$ on the convergence of DCSMD-SW. 
	We choose the stochastic gradient noise to have a uniformly random direction and a Weibull distribution for its norm. This ensures that the stochastic gradient noise adheres to the sub-Weibull condition. 
	Fig. \ref{regression_theta} demonstrates the simulation results for four cases of sub-Weibull parameter $\theta$, it aligns with our theoretical findings, and indicates that the convergence rate of DCSMD-SW deteriorates as the tail parameter $\theta$ increases. Next, it is important to note that DSMD (DCSMD with $\psi_i=0$, $i\in\huaV$) can offer adaptive updates that effectively account for the geometry of the underlying constraint set.  Accordingly, we solve a distributed linear regression problem:
$\min_{x\in\mathcal{X}} \sum_{i=1}^m \frac{1}{2}\left(\langle {a}_{i},{x}\rangle-b_{i}\right)^2$,
where we consider $\huaX=\Delta_n=\{x\in\mbb R^n:\sum_{i=1}^n[x]_i=1, [x]_i\geq0, i\in [n]\}$ (probability simplex), and utilize the Kullback-Leibler divergence $D_\ff(x\|y)=\sum_{i=1}^n[x]_i\ln\f{[x]_i}{[y]_i}$ induced by the Gibbs entropy function $\Phi(x)=\sum_{i=1}^n[x]_i\ln[x]_i$. Accordingly, DCSMD-SW degenerates to  distributed stochastic entropic descent (DSED-SW) (see e.g. \cite{yhhj2018}): $y_{i,t}=\sum_{j=1}^m[W_t]_{ij}x_{j,t}$; $[x_{i,t+1}]_j=\f{[y_{i,t}]_j\exp(-\aaa_t[\wh g_{i,t}^{SW}]_j)}{\sum_{s=1}^n[y_{i,t}]_s\exp(-\aaa_t[\wh g_{i,t}^{SW}]_s)}, \ j\in[n]$.
 The data is generated as previously described, other than that we choose $[x_*]_i=2/n$ for $1 \leq i \leq  n/2 $ ($n$ is even) and $0$ otherwise. The initial values $\{x_{i,1}\}_{i=1}^m$ are drawn from a symmetric Dirichlet distribution with $\alpha=1$, which corresponds to a uniform distribution over the simplex $\Delta_n$.  We illustrate the convergence performance of DSED-SW through a plot of the maximum, minimum, and median optimization errors across all agents against the total number of iterations $T$. The outcomes, depicted in Fig. \ref{figure7}, indicate that DSED-SW achieves a satisfactory convergence rate.

\begin{figure}[t] 
		\centering
	\begin{minipage}{0.34\textwidth}
		\centering
		\includegraphics[width=\linewidth]{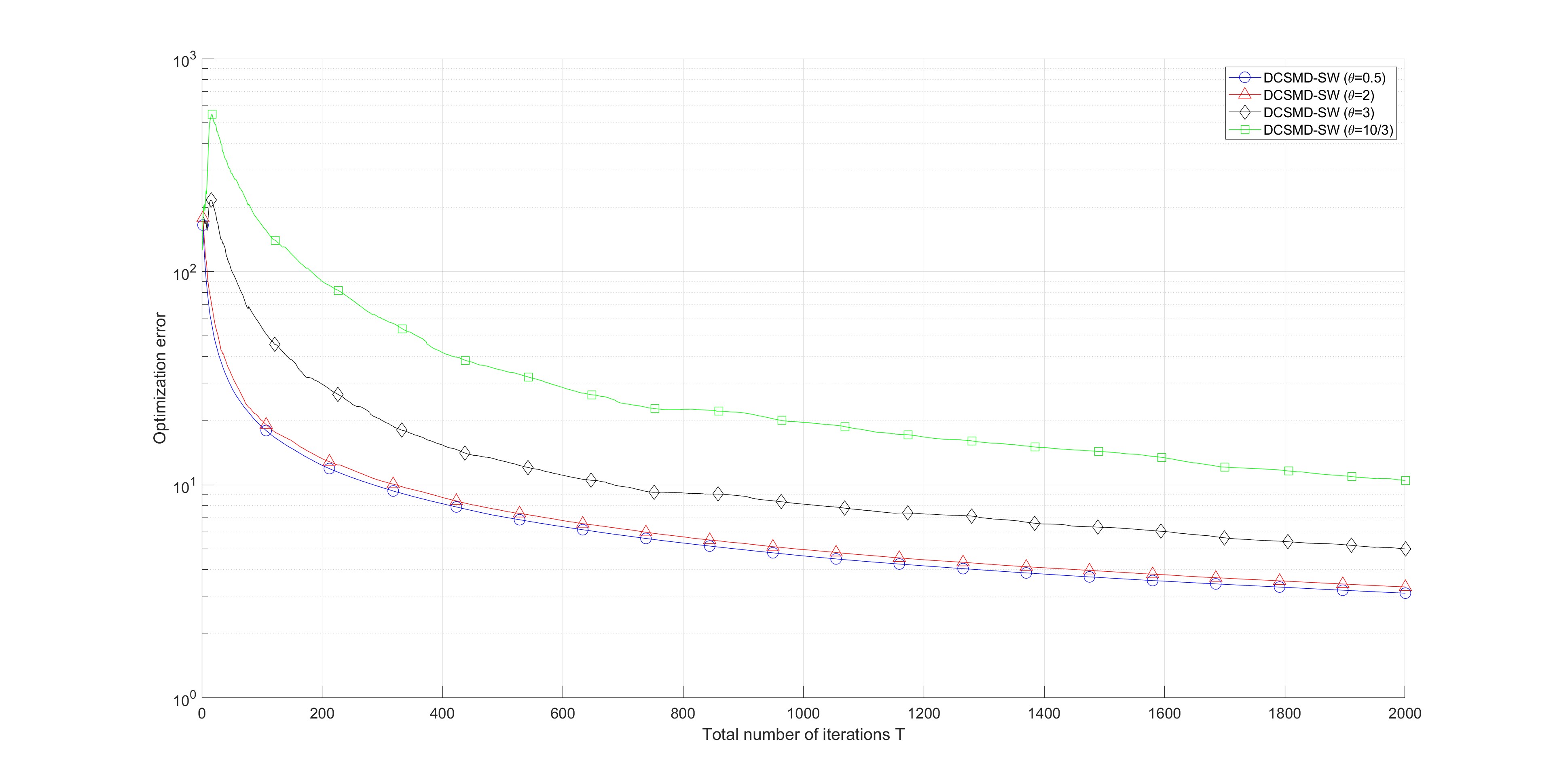}
		\caption{\tiny{Performance of DCSMD-SW with different tail parameters $\theta$ in regularized logistic regression with unbounded decision space.}}
		\label{class_theta}
	\end{minipage}
	\centering
	\begin{minipage}{0.34\textwidth}
		\centering
		\includegraphics[width=\linewidth]{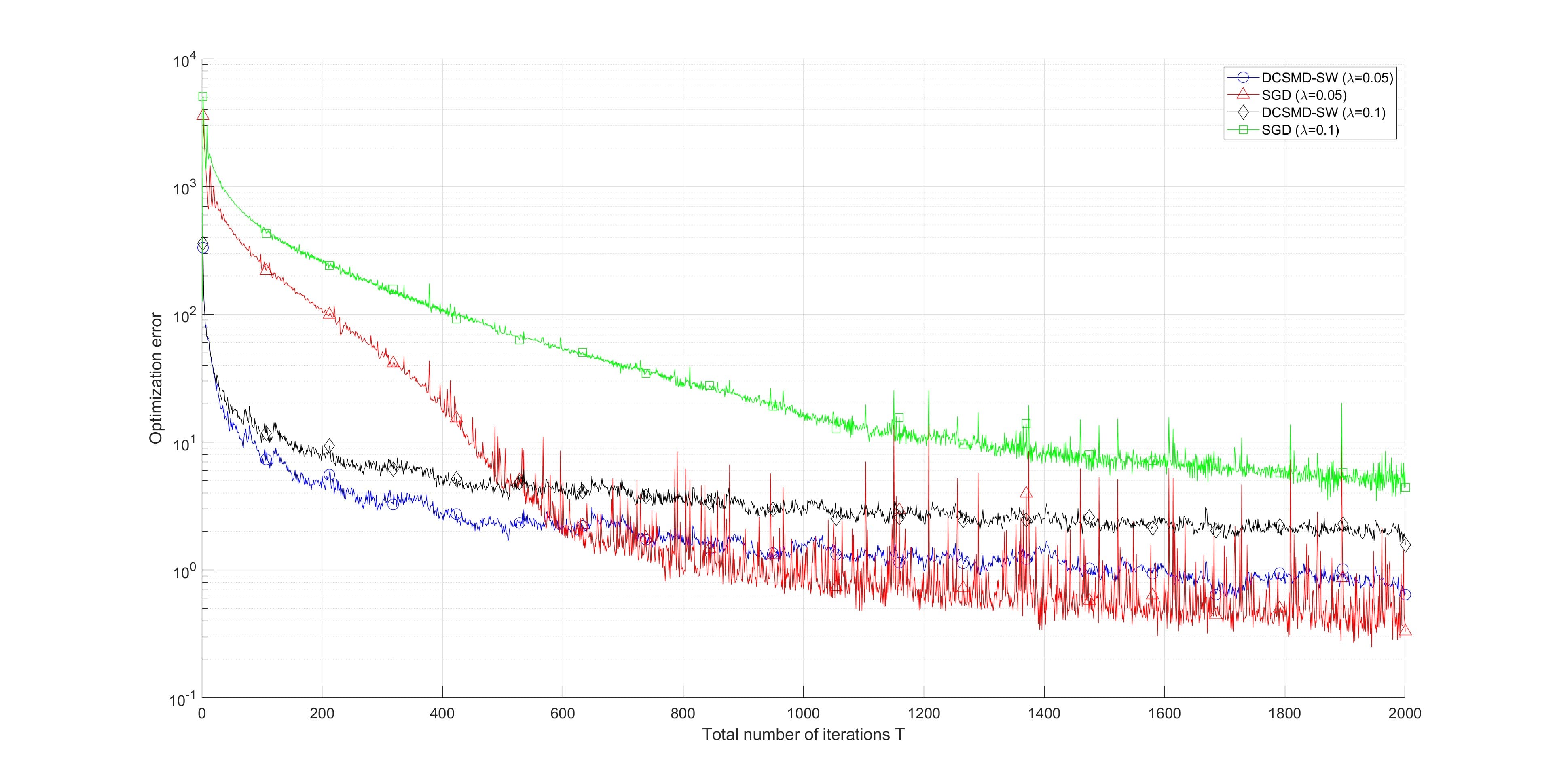}
		\caption{\tiny{Performance of DCSMD-SW and SGD with different $\lambda$ under sub-Weibull noises in regularized logistic regression.}}
		\label{sgd}
	\end{minipage}
	\hfill
	\begin{minipage}{0.34\textwidth}
		\centering
		\includegraphics[width=\linewidth]{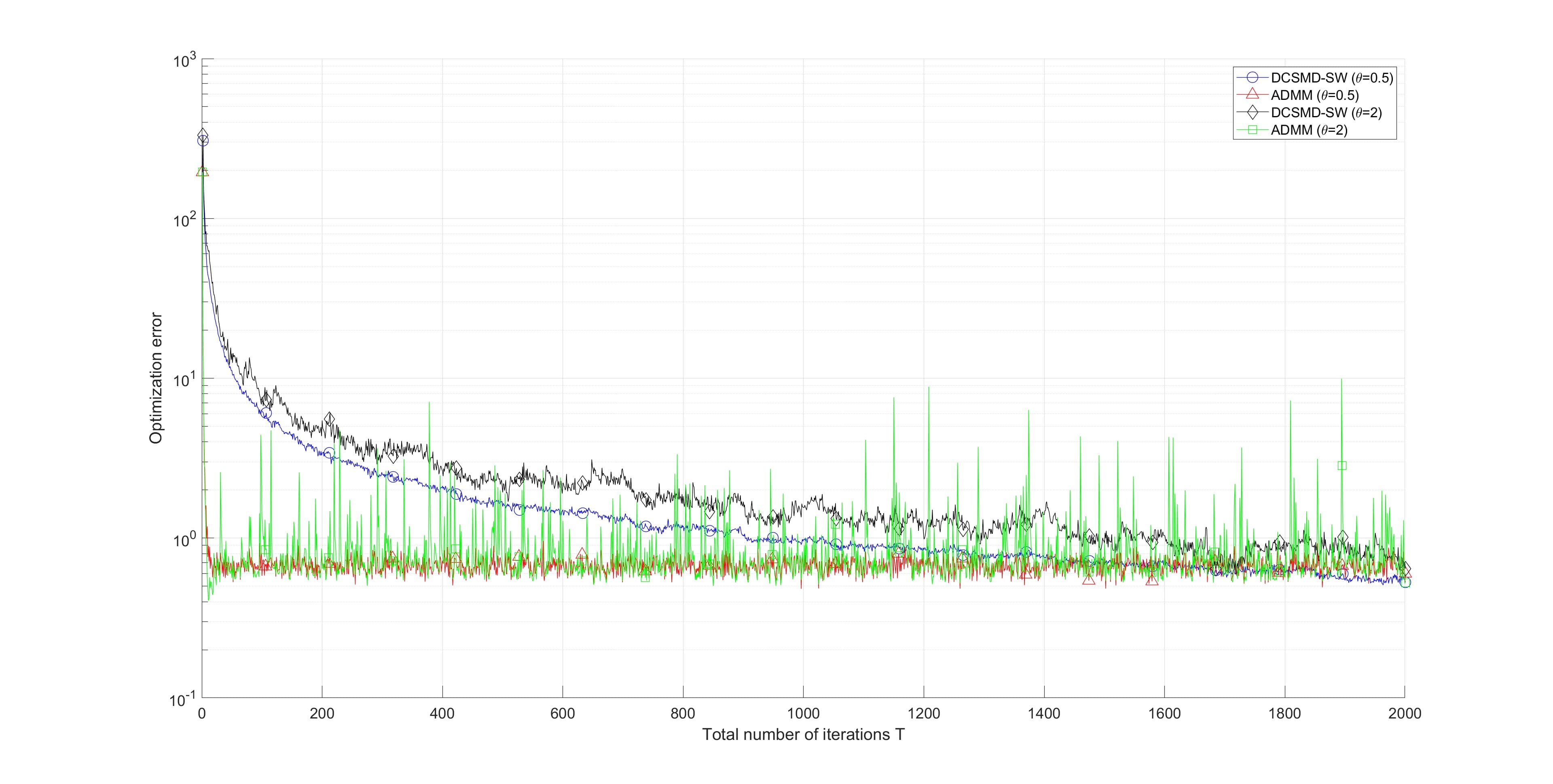}
		\caption{\tiny{Performance of DCSMD-SW and  SADMM under sub-Weibull noises with different $\theta$ in regularized logistic regression.}}
		\label{admm}
	\end{minipage}
\end{figure}

Now we consider a distributed regularized logistic regression problem: 
	\bes
		\min_{x\in\mbb R^n} \sum_{i=1}^m \Big(\sum_{j=n_{i-1}+1}^{n_i} \log\left(1+\exp(-b_j \cdot \langle {a}_{j},{x}\rangle)\right)+\lambda\| x\|_1\Big),
	\ees
	the data is generated as previously described, other than $b_i = \text{sign}(\langle a_i, x_* \rangle + \epsilon_i)$. In this problem, the decision space is unbounded, and the  logistic loss satisfies all basic assumptions on objective function of this paper. We demonstrate the influence of the tail parameter $\theta$ in sub-Weibull stochastic gradient noise on the performance of DCSMD-SW by graphing the median of the optimization errors 
	across $m$ agent, plotted against the total number of iterations $T$ in Fig. \ref{class_theta}. It indicates that the convergence rate of DCSMD-SW deteriorates as the tail parameter $\theta$ increases as well in this problem.

We also compare DCSMD-SW with some other classical algorithms in this problem under sub-Weibull noises, i.e. SGD and SADMM. We choose sub-Weibull parameter $\theta=2$ and stepsizes for DCSMD-SW and SGD as $\alpha_k= 1/\sqrt{k+1}$. In Fig.~\ref{sgd}, we present the median optimization errors of DCSMD-SW across $m$ agents and the optimization errors of SGD under two different choices of the regularization parameter $\lambda$. The numerical results suggest that DCSMD-SW may perform slightly better than SGD in certain cases, indicating that DCSMD-SW has the potential to encompass or recover the performance of centralized algorithms. In Fig.~\ref{admm}, we set the regularization parameter to $\lambda= 0.02$, and present the median optimization errors of DCSMD-SW across $m$ agents and the optimization errors of SADMM utilized in \cite{blz2021} under different choices of the sub-Weibull parameter $\theta$ in the stochastic gradient noise. The penalty parameter $\beta$, the linearization parameter $\tau$ and the stepsize $\sigma$ for Lagrange multiplier update of SADMM in \cite{blz2021} are chosen as $\sigma=\beta=1$, $\tau=0.025$. The numerical results indicate that SADMM converges more rapidly during the initial stage, whereas DCSMD-SW achieves performance comparable to SADMM as the number of iterations increases.

\section{Conclusions and discussions}
The convergence theory of DCSMD-SW has been comprehensively established. The explicit convergence rate is fully derived for two types of ergodic approximating sequences generated from DCSMD-SW under appropriate selection rules of stepsizes. As a byproduct of this work, the convergence behaviors of  DSGD under   sub-Weibull randomness has been directly obtained. Notably, the convergence characteristics of DSGD-SW have not been independently studied in existing literature. Moreover, this work eliminates the need for the crucial technical assumption of boundedness of the decision set, which has been a dependency in existing research on DMD approach for convex optimization. Numerical experiments are conducted to support the theory. It would be interesting to extend our analysis to other popular occasions such as distributed Nash equilibrium seeking, strongly convex DO and online DO. It would be interesting to study the non-ergodic convergence of DCSMD-SW. We leave them for future work.
\section*{Acknowledgment}
The authors would like to thank the AE and the four anonymous Referees for their  constructive and inspiring comments that help improve the quality of the article.

\end{CJK*}
\end{document}